\documentclass[12pt,english]{article}
\usepackage[T1]{fontenc}
\usepackage[utf8]{inputenc}
\usepackage[english]{babel}
\usepackage{latexsym}
\usepackage{url}
\usepackage{amsmath}
\usepackage{amssymb}
\usepackage{amsfonts}
\usepackage{graphicx}
\usepackage{float}
\usepackage{subcaption}
\usepackage{upgreek}
\usepackage{ mathrsfs }
\usepackage{ dsfont }
\usepackage{amsthm}
\usepackage{graphicx}
\usepackage{caption}
\usepackage{booktabs}
\usepackage{color}
\usepackage{framed}
\usepackage{verbatim}
\usepackage{pdflscape,rotating,arydshln}
\usepackage{bbm}
\usepackage{cancel}

\usepackage[unicode=true,psdextra,
 bookmarks=true,bookmarksnumbered=true,bookmarksopen=false,
 breaklinks=true,pdfborder={0 0 1},backref=false,colorlinks=true]
 {hyperref}
 
\hypersetup{pdftitle={Generalized Wright Analysis: Stochastic and Applications},
 pdfauthor={Lorenzo Cristofaro, Jos\'e. L.~da Silva, W.~Bock},
 pdfsubject={Fox-H Functions, non-Gaussian measures, generalized Wright functions, stochastic processes},
 pdfkeywords={Fox-H Functions, non-Gaussian measures, generalized Wright functions, stochastic processes},
 linkcolor=black,citecolor=black,urlcolor=black,filecolor=black}
\definecolor{shadecolor}{RGB}{248,248,248}

\theoremstyle{plain}
\newtheorem{theorem}{Theorem}[section]
\theoremstyle{definition}
\newtheorem{definition}[theorem]{Definition}

\newtheorem{remark}[theorem]{Remark}
\newtheorem{corollary}[theorem]{Corollary}
\newtheorem{proposition}[theorem]{Proposition}
\newtheorem{lemma}[theorem]{Lemma}
\newtheorem{assumption}[theorem]{Assumption}
\newtheorem{property}[theorem]{Property}
\newcommand{\N}{\mathbb{N}}
\newcommand{\Z}{\mathbb{Z}}
\newcommand{\R}{\mathbb{R}}
\newcommand{\C}{\mathbb{C}}
\newcommand{\e}{\mathrm{e}}

\begin{document}
\title{Generalized Wright Analysis: Stochastic and Applications}
\author{\textbf{Wolfgang Bock,}\\
 Department of Mathematics,\\
 Linnaeus University,\\
 Växjö, 351 95, Växjö, Sweden.\\
 Email: wolfgang.bock@lnu.se. \and
 \textbf{Lorenzo Cristofaro}\\
 Department of Mathematics,\\
 University of Luxembourg\\ 
 Esch-sur-Alzette, Luxembourg,\\
 Email: lorenzo.cristofaro@uni.lu \and
 \textbf{Jos{é} L.~da Silva},\\
 Faculty of Exact Sciences and Engineering,\\
 CIMA, University of Madeira, Campus da Penteada,\\
 9020-105 Funchal, Portugal.\\
 Email: joses@staff.uma.pt}
\date{\today}

\maketitle

\begin{abstract}

In this paper, we investigate the stochastic counterpart of the generalized Wright analysis introduced in Beghin et al.~ in Integral Equations and Operator Theory, {\bf 97}, 2025. We define a new class of non-Gaussian and non-Markovian processes, called the generalized Fox-$H$ process, which extends well-known processes such as fractional Brownian motion and generalized grey Brownian motion. We study their joint probability density and covariance, showing the stationarity of their increments. In addition, this process has H{\"o}lder continuous paths and is represented as a time-change of fractional Brownian motion. 
We characterize the generalized Fox-$H$ noise as an element in the distribution space $(\mathcal{S})^{-1}_{\mu_\Psi}$. We conclude by establishing the existence of local times and discussing their anomalous diffusion properties.
\\[.2cm]
\\[.2cm]
\noindent \emph{AMS Subject Classification 2010}: fractional Brownian motion, time-change, distribution, local times, noise of processes, special functions, Mittag-Leffler, generalized Wright functions, integral equation, integral transforms.
\end{abstract}

	\tableofcontents
	
	\newpage

 \section{Introduction} 

Since its introduction, white noise analysis has been applied to many scientific fields. The calculus uses white noise, the derivative of Brownian motion within a particular distribution space, to represent observables. Indeed, the theory combines modern concepts of harmonic analysis and probability theory to and has been used for important results in physics, finance, and other fields.

More specifically, the white noise space is the (probability) measure space $(\mathcal{S}'(\R), \mathcal{C}_\sigma(\mathcal{S}'(\R)), \mu)$, where $\mathcal{S}'(\R)$ is the dual space of the Schwartz space of rapidly decreasing smooth functions,  $\mathcal{C}_\sigma(\mathcal{S}'(\R))$ is the $\sigma$-algebra generated by cylinder sets, and $\mu$ is the Gaussian white noise measure. Other probability measures can be considered on $\mathcal{S}'(\R)$ or on another nuclear triple $\mathcal{N}\subset \mathcal{H}\subset\mathcal{N}'$ to cover a calculus for point processes or more generally, non-Gaussian measures. We refer to the following books for Gaussian analysis \cite{HKPS93,BK95,Oba,RS72} and \cite{Ito88,AKR97a,KDSSU98} for Poisson and related measures.

New classes of non-Gaussian measures were defined (on an abstract nuclear triple $\mathcal{N}\subset\mathcal{H}\subset\mathcal{N}'$) by employing elementary or special functions in the characteristic functions; see \cite{BCDS24, Beghin2022, JahnI} and references therein. In this paper, we use the class introduced in \cite{BCDS24}. There, as in the white noise case random variables, and hence stochastic processes, are realized by the dual pairing between elements of $\mathcal{N}'$ and $\mathcal{N}$. Via an Ito-isometry argument, they can be extended to pairings between elements of $\mathcal{N}'$ and $\mathcal{H}$. By considering a particular nuclear triple, where the central Hilbert space is an $L^2$ function space, we can define a class of typically non-Markovian stochastic processes, such as fractional Brownian motion (fBm) and time-changed fBm, among others. 

There is a standard way to construct test and generalized function spaces associated with a class of non-Gaussian measures. More precisely, those measures that have an analytic Laplace transform in a neighborhood of zero and have positive measure on non-empty open sets; see \cite{KSWY98} and \cite{KK99} for details. Objects such as Donsker's delta, local time, stochastic currents, and solutions of stochastic differential equations are well-defined as generalized functions in this framework.

Examples of non-Gaussian measures and related analytical frameworks are given in \cite{Beghin2022, BCP26, DSO12}. In this regard, \cite{Sch88} and \cite{Sch90} introduced a non-Gaussian family of probability measures called grey noise measures, whose characteristic functionals are given by Mittag-Leffler functions.

Moreover, the latter turned out to be the characteristic function of a non-Markovian process called generalized grey Brownian motion, whose probability density function is the fundamental solution of a "stretched" time-fractional operator; see \cite{Sch90, MMP10}.
Later on, Grothaus et al.~\cite{JahnI} extended the family of grey noise measures to 
Mittag-Leffler measures on an abstract nuclear triple $\mathcal{N}\subset\mathcal{H}\subset\mathcal{N}'$. This family of measures $\mu_\beta$, $0<\beta\le1$, on $(\mathcal{N}', \mathcal{C}_{\sigma}(\mathcal{N}'))$ satisfies 
\[ 
\int_{\mathcal{N}'} \e^{\mathrm{i} \langle \omega, \xi \rangle} \mathrm{d}\mu_\beta(\omega)=E_{\beta}\left( -\frac{1}{2}\langle \xi, \xi \rangle   \right), \qquad \xi \in \mathcal{N},  
\]
where the Mittag-Leffler function $E_{\beta}(x):=\sum_{j = 0}^\infty \frac{1}{\Gamma(\beta j +1)}x^j$, $x \in \R$.
Another goal of non-Gaussian analysis is to give a relation between general differential equations and stochastic processes. In \cite{JahnII}, specific non-Gaussian distribution spaces are shown to be required to properly define the solution of the (potential free) fractional heat equation and to represent it using a Feynman-Kac formula. Furthermore, the interplay between anomalous diffusions and fractional diffusion equations carries numerous scientific applications, such as relaxation-type differential equations, continuous time random walks, and viscoelasticity; see \cite{Koc11, MS19,Mai22}.\\
In this paper, we introduce a large class of stochastic processes called the generalized Fox-$H$ process $X_t^{\mathbbm{H},A,B}$, $t\ge0$ (abbreviated as gFHp in the following), study their properties, and explore their first applications. More precisely, we study the characterize the noise process, local times, and the related anomalous diffusion. Other applications, such as stochastic differential equations driven by gFHp noise and Green measures, will be studied in a forthcoming paper.\\
The key role of the Mittag-Leffler function in the above infinite-dimensional setting is replaced here, for the first time, by the use of generalized Wright functions $_m\Psi_p(\cdot)$, hereinafter referred to as gWf, with the Mittag-Leffler function being a special case.
More specifically, this extended setting is called generalized Wright analysis, and it is the non-Gaussian probability space denoted by $(\mathcal{N}',\mathcal{C}_\sigma(\mathcal{N}'), \mu_{\Psi} )$ and defined by the use of gWf.

The paper is organized as follows. In Section~\ref{sec:gwa}, we recall the definition and key properties of the class of measures of interest from \cite{BCDS24}. We explicitly compute the moments of this class of measures, introduce generalized stochastic processes, and define the test and distribution spaces associated with it. As an application, we characterize the strong convergence of sequences in suitable distribution spaces. We conclude this section by recalling the Radon-Nikodym derivatives for the finite-dimensional case and expressing this class of measures as a mixture of Gaussian measures with a family of densities in $\mathbb{R}^+$. In Section~\ref{Sec 3: Gene Stoch Proce}, we recall the theory of generalized random variables and fractional calculus to introduce gFHp as a special case. We study their properties, show the finite-dimensional distribution, and present the representations in terms of a positive random variable and a fractional Brownian motion. We conclude this section by establishing the existence of the gFHp noise process as a well-defined element in a distribution space. Section~\ref{sec:applications} provides applications of the study of local times associated with gFHp and demonstrates slow and fast diffusion for certain classes of this family. In the outlook and discussion, we identify the next lines of research within this framework.

 \section{Generalized Wright Analysis\label{sec:gwa}}  
In this section, we recall the class of generalized Wright measures as probability measures on a general conuclear space $\mathcal{N}'$ (see \cite{BCDS24} for more details) and the essential properties required for what follows.
 
 \subsection{Preliminaries}
 To define the desired class of measures, we introduce two families of functions: Fox-$H$ functions and the generalized Wright function $_m \Psi_p(\cdot)$. The parameter assumptions ensure that $_m \Psi_p(\cdot)$ becomes a characteristic function and is given as a mixture of Gaussian characteristic functions by Fox-$H$ densities. 
 
\begin{definition}[cf.~p.1 in \cite{SaiKil}]\label{def:H-function}
For integers $m,n,p,q$ such that $0 \leq m \leq q$, $0 \leq n \leq p$, for $a_i, b_j \in \C$ and for $\alpha_i,\beta_j \in (0,\infty)$ with $i=1,\dots,p $ and $j=1,\dots,q$, the Fox-$H$ function $H_{p,q}^{m,n}(\cdot)$ is defined via a Mellin-Barnes integral as
	
	\[ 
	H_{p,q}^{m,n}(z):=H_{p,q}^{m,n}\left[z \, \middle| \genfrac{}{}{0pt}{}{(a_i,\alpha_i)_p}{(b_j,\beta_j)_q}\right]:=\frac{1}{2\pi \mathrm{i}}\int_{\mathcal{L}} \mathcal{H}^{m,n}_{p,q}(s)z^{-s}\,\mathrm{d}s, \quad z \in \mathbb{C}\backslash \{0\},
	\]
	where 
	\[  
	\mathcal{H}^{m,n}_{p,q}(s):= \frac{\prod_{j=1}^{m} \Gamma({b_j+s\beta_j}) \prod_{i=1}^{n}\Gamma({1-a_i-s\alpha_i})}{ \prod_{i=n+1}^{p}\Gamma({a_i+s\alpha_i})\prod_{j=m+1}^{q} \Gamma({1-b_j
			-s\beta_j})}
	\]
	and $\mathcal{L}$ is the infinite contour that separates all the poles given by Equation~\eqref{def:polesb} from those given by Equation~\eqref{def:polesa} below.

\end{definition}
The function $\mathcal{H}^{m,n}_{p,q}(\cdot)$ inherits the poles from the functions $\Gamma(b_j+s\beta_j)$
\begin{equation} \label{def:polesb} b_{jl}=-\frac{b_j+l}{\beta_j}\quad j=1,\dots,m, \quad l=0,1,2,\dots \end{equation}
and the functions $\Gamma(1-a_i-s\alpha_i)$
\begin{equation} \label{def:polesa} a_{ik}=\frac{1-a_i+k}{\alpha_i}\quad i=1,\dots,n,\quad k=0,1,2,\dots. \end{equation}

The above poles do not coincide if $\alpha_i(b_j+l)\neq \beta_j(a_i-k-1)$, for $i=1,\dots,n$,  $j=1,\dots,m$, with $k,l \in \N_0$; see Equation~(1.1.6) in \cite{SaiKil}. \\
The contour $\mathcal{L}$ can take three distinct shapes, which we refer to as $\mathcal{L}_{-\infty}$, $\mathcal{L}_{+\infty}$, and $\mathcal{L}_{\mathrm{i}\gamma\infty}$, defined as follows:
\begin{itemize}
	\item $\mathcal{L}=\mathcal{L}_{-\infty}$ is a left loop situated in a horizontal strip starting at the point $-\infty + \mathrm{i} \phi_1$ and terminating at the point $-\infty + \mathrm{i} \phi_2$ with $-\infty<\phi_1<\phi_2<+\infty$;
	\item $\mathcal{L}=\mathcal{L}_{+\infty}$ is a right loop situated in a horizontal strip starting at the point $+\infty + \mathrm{i} \phi_1$ and terminating at the point $+\infty + \mathrm{i}\phi_2$ with $-\infty<\phi_1<\phi_2<+\infty$;
	\item $\mathcal{L}=\mathcal{L}_{\mathrm{i}\gamma\infty}$ is a contour starting at the point $\gamma -\mathrm{i}\infty$ and terminating at the point $\gamma +\mathrm{i}\infty$, where $\gamma \in \mathbb{R}=(-\infty,+\infty)$.
\end{itemize}
Here $z^{-s}=\exp[-s(\mathrm{Log}(|z|) + \mathrm{i}\arg(z))]$, $z \neq0$, where $\mathrm{Log}(|z|)$ represents the natural logarithm of $|z|$ and $\arg(z)$ is the principal value. 

The properties of the Fox-$H$ function depend on the contour $\mathcal{L}$ and the quantities $a^{*}, \Delta, \delta$ and $\mu$, which are defined as
\begin{eqnarray*}
    a^{*} &:=& \sum_{j=1}^m\beta_j+\sum_{i=1}^n\alpha_i-\left(\sum_{j=m+1}^q\beta_j +\sum_{i=n+1}^p\alpha_i\right), \\
    \Delta &:=& \sum_{j=1}^q\beta_j-\sum_{i=1}^p\alpha_i,   \\
    \delta &:=& \prod_{j=1}^q\beta_j^{\beta_j} \prod_{i=1}^p\alpha_i^{-\alpha_i}, \\
    \mu &:=& \sum_{j=1}^q b_j-\sum_{i=1}^p a_i +(p-q)/2,
\end{eqnarray*}
see Equations~(1.1.9), (1.1.10), (1.1.11) in \cite{SaiKil}.
Any empty sum is taken to be zero, and any empty product is taken to be $1$.

We recall the definition of the generalized Wright functions $_m\Psi_p(\cdot)$ (a subclass of Fox-$H$ functions also known as Fox-Wright functions) that represents the largest class of special functions that could be expressed as the Laplace transform of densities using Lemma~\ref{lem:HFiniteDensities} below.

For $m,p \in \N_0:=\{0,1,2,\dots\}$, $A_i,B_j\in \mathbb{C}$ and $\alpha_i,\beta_j \in \R \backslash \{0\}$, with $i=1,\dots,m$ and $j=1,\dots,p$, the generalized Wright  functions are defined by the series 
\[ 
_m\Psi_p\left[ \genfrac{}{}{0pt}{}{(A_i, \alpha_i )_{1,m}}{(B_i, \beta_j )_{1,p} }\middle|\, z \right]=\sum_{k\geq0} \frac{\prod_{i=1}^m \Gamma(A_i+\alpha_i k)}{\prod_{j=1}^p \Gamma(B_j+\beta_j k)} \frac{z^k}{k!}, \quad z \in \mathbb{C}.
\]
 They are represented by Fox-$H$ functions if $\alpha_i,\beta_j$ are positive:
 \[  
 _m\Psi_p\left[ \genfrac{}{}{0pt}{}{(A_i, \alpha_i )_{1,m}}{(B_i, \beta_j )_{1,p} }\Bigg| -z \right]=H^{1,m}_{m,p+1}\left[z \, \Bigg| \genfrac{}{}{0pt}{}{(1-A_i, \alpha_i )_{1,m}}{(0,1),(1-B_j, \beta_j )_{1,p} }\right], \quad z \in \mathbb{C} ,
 \]
 see Equation (5.2) in \cite{KST02}.

We will use this alternative explicit notation when needed:
\[  
H_{p,q}^{m,n}\left[z \, \Bigg| \genfrac{}{}{0pt}{}{(a_i,\alpha_i)_{1,p}}{(b_j,\beta_j)_{1,q}}\right]=H_{p,q}^{m,n}\left[z \,\Bigg| \genfrac{}{}{0pt}{}{(a_1,\alpha_1),\dots, (a_p,\alpha_p)}{(b_1,\beta_1),\dots,(b_q,\beta_q)}\right].
\]
In the case $p=1$ or $q=1$ we will use the short notation: $(a_i,\alpha_i)_{1,p}=(a,\alpha)$ or $(b_j,\beta_j)_{1,q}=(b,\beta)$, respectively. 

For gWfs, we will use the following short notation when necessary:
\[ 
_m\!\Psi_p(\cdot):= \, _m\!\Psi_p\left[ \genfrac{}{}{0pt}{}{(b_i+\beta_i,\beta_i)_{1,m}}{(a_j+\alpha_j,\alpha_j)_{1,p}} \middle|\, \cdot \right] .
\]

 The following assumptions fix the constraints on the parameters $a_i$, $b_j$, $\alpha_i$, $\beta_j$, $i=1,\dots,p$, $j=1,\dots,m$ and the values of Fox-$H$ functions we are considering. In addition, we will reduce the third one to be $a^* \in (0,1)$, which guaranties the analyticity of the gWfs on the entire complex plane. Indeed, in Assumption~\ref{ass:AllHdensity} below, we specify the conditions under which the Fox-$H$ densities $\varrho_H(\cdot)$ possess all moments. The densities $\varrho_H(\cdot)$, known as F$H$dam, characterize gWfs, as outlined in Lemma~\ref{lem:HFiniteDensities} below.
 
 \begin{assumption}\label{ass:AllHdensity}
    Let $m,p \in \mathbb{N}_0$ be given. We make the following assumptions on the parameters of the Fox-$H$ function:
\begin{enumerate}
\item Let $a_i \in \mathbb{R}$ and $\alpha_i>0$, for $i=1,\dots,p$, and $a_i+\alpha_i>0$, for $i=1,\dots,p$;
\item Let $b_j \in \mathbb{R}$ and $\beta_j>0$, for $j=1,\dots,m$,  and $b_j + \beta_j>0$, for $j=1,\dots,m$;
\item Let either $a^*>0$ or $a^*=0$ and $\mu<-1$;
\item Let $H^{m,0}_{p,m} \left[ \cdot \,\big|\genfrac{}{}{0pt}{}{(a_i,\alpha_i)_{1,p}}{(b_j,\beta_j)_{1,m}} \right]$ be non-negative on $(0,\infty)$.
\end{enumerate}

   In 3.~$a^*:=\sum_{j=1}^m \beta_j - \sum_{i=1}^p \alpha_i$ and $\mu:=\sum_{j=1}^m b_j -\sum_{i=1}^p a_i - (p-m)/2$.
\end{assumption}
\begin{lemma}[Fox-$H$ densities with all moments, see Lemma~5 in \cite{BCDS23}]\label{lem:HFiniteDensities}
			Let the Assumption~\ref{ass:AllHdensity} holds for $n=0$ and $q=m$. Then the corresponding Fox-$H$ density	on $[0,\infty)$
\begin{equation}\label{eq:densitiesFiniteMoments}  
\varrho_H(\tau):=\frac{1}{K}H^{m,0}_{p,m} \left[ \tau \,\Bigg|\genfrac{}{}{0pt}{}{(a_i,\alpha_i)_{1,p}}{(b_j,\beta_j)_{1,m}} \right], \quad \tau>0,\end{equation}
where 
\begin{equation}\label{eq:norm-constant}
K:=\frac{\prod_{i=1}^m \Gamma(b_i+\beta_i)}{\prod_{i=1}^p \Gamma(a_i+\alpha_i)}     \end{equation} 
has finite moments of all orders. The moments are given by 
\[
\int_{0}^\infty \tau^l  \varrho_H(\tau)\, \mathrm{d}\tau
=\frac{1}{K} \frac{\prod_{i=1}^m \Gamma(b_i+\beta_i(l+1))}{\prod_{i=1}^p \Gamma(a_i+\alpha_i(l+1))}, \quad l=0,1,\dots.
\]
Furthermore, its Laplace transform is given by
\begin{equation}\label{eq:LTdensitiesFiniteMoments} 
(\mathscr{L}\varrho_H)(s)
=\frac{1}{K}\, _m\!\Psi_p\left[ \genfrac{}{}{0pt}{}{(b_i+\beta_i,\beta_i)_{1,m}}{(a_j+\alpha_j,\alpha_j)_{1,p}} \middle| -s\right],   \end{equation}
where $s\geq 0$. Furthermore, when $a^* \in (0,1)$, $(\mathscr{L}\varrho_H)(\cdot)$ can be extended to an entire function.\\  
The density $\varrho_H(\cdot)$ mentioned above is referred to as the Fox-$H$ density with all moments finite, abbreviated as F$H$dam. The set of all random variables whose distributions possess an F$H$dam density with respect to the Lebesgue measure is denoted by $\mathcal{X}$.
\end{lemma}

\begin{remark}\label{rem:classes-rvariables}
The probabilistic interpretation of F$H$dam densities in terms of well-known probability density functions is provided in Theorem~8 and Remark~10 in \cite{BCDS23}. Indeed, we recall that eight classes of F$H$dam are indexed as $\mathit{(C0)}$-$\mathit{(C7)}$ and described in terms of well-known density functions as beta, gamma, and the so-called generalized $M$-Wright, defined as follows.
The generalized $M$-Wright density function $\varrho_{\mathit{(C4)}}(x)$, $x \in (0,\infty)$, with parameters $\beta \in (0,1)$, $a \in \mathbb{R}$, $\alpha>0$ such that $a+\alpha>0$, is given as 
\begin{equation}\label{eq:gMWrightdensity}  
\varrho_{\mathit{(C4)}}(x):=\frac{\Gamma(1-\beta + \beta a +\beta \alpha)}{\Gamma(a+\alpha)} H^{1,0}_{1,1}\left[ x\, \Bigg| \genfrac{}{}{0pt}{}{(1-\beta+\beta a,\beta\alpha)}{(a, \alpha)} \right]. 
\end{equation}
For $a=0$ and $\alpha=1$, the density $\varrho_{\mathit{(C4)}}(\cdot)$ reduces to the well-known $M_{\beta}$ function, see Equation~(4.9) in \cite{MPS05}.
The $\mathit{(C0)}$-$\mathit{(C7)}$ F$H$dam densities are understood as:
\begin{itemize}
\item The class $\mathit{(C0)}$ corresponds to the degenerate random variable with distribution $\delta_1$.
\item The class $\mathit{(C1)}$ describes the class of gamma random variables, their product and powers. 
\item The class $\mathit{(C2)}$ is similar to $\mathit{(C1)}$ for beta random variables.
\item The class $\mathit{(C3)}$ is obtained as the product of classes $\mathit{(C1)}$ and $\mathit{(C2)}$.
\item The class $\mathit{(C4)}$ is associated with generalized $M$-Wright random variables, its products, and powers; see Equation~\eqref{eq:gMWrightdensity}.
\item The class $\mathit{(C5)}$ is obtained as the product of classes $\mathit{(C1)}$ and $\mathit{(C4)}$.
\item The class $\mathit{(C6)}$ is obtained as the product of classes $\mathit{(C2)}$ and $\mathit{(C4)}$.
\item The class $\mathit{(C7)}$ is obtained as the product of the classes $\mathit{(C3)}$ and $\mathit{(C4)}$, or, equivalently, as the product of the classes $\mathit{(C1)}$, $\mathit{(C2)}$, and $\mathit{(C4)}$.
\end{itemize}
\end{remark}

\begin{remark}\label{rem:GeneralizedWrightDerivatives}
We report here the derivatives of gWfs. Indeed, we recall that $_m\Psi_p(-z)$, $z\in \mathbb{C}$, are entire functions under the Assumption~\ref{ass:AllHdensity} on the parameters $a_i,b_j,\alpha_i,\beta_j$ for $i=1,\dots,p$ and $j=1,\dots,m$ together with $a^*\in(0,1)$. We have
\begin{eqnarray*}\label{eq:GWfunctionDerivative} 
    \frac{\mathrm{d}}{\mathrm{d} x} \, _m\Psi_p\left[ \genfrac{}{}{0pt}{}{(B_i, \beta_i )_{1,m}}{(A_j, \alpha_j )_{1,p} }\Bigg| x \right]&=&\frac{\mathrm{d}}{\mathrm{d} x}\sum_{k\geq0} \frac{\prod_{i=1}^m \Gamma(B_i+\beta_i k)}{\prod_{j=1}^p \Gamma(A_j+\alpha_j k)} \frac{x^k}{k!}\\
    &=& \sum_{k\geq1} \frac{\prod_{i=1}^m \Gamma(B_i+\beta_i k)}{\prod_{j=1}^p \Gamma(A_j+\alpha_j k)} \frac{x^{k-1}}{(k-1)!}\\
    &=& \sum_{k\geq0} \frac{\prod_{i=1}^m \Gamma(B_i+\beta_i+\beta_i k)}{\prod_{j=1}^p \Gamma(A_j+\alpha_j+\alpha_j k)} \frac{x^{k}}{k!}\\
    &=&_m\Psi_p\left[ \genfrac{}{}{0pt}{}{(B_i+\beta_i, \beta_i )_{1,m}}{(A_j+\alpha_j, \alpha_j )_{1,p} }\Bigg| x \right].
\end{eqnarray*}
Recalling the assumptions on the parameters and the properties of gWds, we highlight that:
\begin{enumerate}
    \item Derivatives of $_m\Psi_p(\cdot)$, resp.~$H^{1,m}_{m,p+1}(\cdot)$, can be formally expressed by a gWf, resp.~Fox-$H$ function;
    \item Derivatives of the functions $_m\Psi_p(\cdot)$ and $H^{1,m}_{m,p+1}(\cdot)$ do not affect the values of $a^*$ and $\Delta$ as well as the contour plot used to define the corresponding Fox-$H$ function.
\end{enumerate}
\end{remark}
Before concluding, we state the behavior of F$H$dam densities near zero and at infinity under some regularity conditions.
\begin{corollary}[See Corollary~7 in \cite{BCDS23}]\label{prop:AsymptotHDensity}
    Let the assumptions of Lemma~\ref{lem:HFiniteDensities} hold, with $a^*>0$ and the poles $b_{jl}$ in Equation \eqref{def:polesb} being simple. Then we have
    \begin{equation}\label{eq:AsymptoticHDensityinfinity} 
    H^{m,0}_{p,m}(x)\sim O\left(x^{(\Re(\mu)+1/2)/\Delta} \exp(-\Delta \delta^{-1/\Delta} x^{1/\Delta})\right), \quad x \to \infty    
    \end{equation}
    and
    \begin{equation}\label{eq:AsymptoticHDensityzero} 
    H^{m,0}_{p,m}(x)\sim O\left(x^{\rho}\right), \quad x \to 0^+,   
    \end{equation}
    where 
    \[  
    \rho:=\min_{j=1,\dots,m}\left[ \frac{b_{j}}{\beta_{j}} \right]>-1. 
    \]
\end{corollary}

 \subsection{The Generalized Fox-\texorpdfstring{$H$}{H} Measure}\label{subsec:gFHm}
To introduce the gFHm, let us recall briefly the notion of a nuclear Gelfand triple. For details, see e.g., \cite{SCH71,RS72}.

Let $\big(\mathcal{H},(\cdot,\cdot)\big)$ be a real separable Hilbert space with associated norm $|\cdot|_0$. In addition, let $\mathcal{N}$ be a nuclear space continuously and densely embedded in $\mathcal{H}$, and denote by $\mathcal{N}'$ the dual of $\mathcal{N}$. The canonical dual pairing between $\mathcal{N}'$ and $\mathcal{N}$ is represented by $\langle \cdot, \cdot \rangle$. Identifying $\mathcal{H}$ with its dual space $\mathcal{H}'$ (via the Riesz isomorphism), we obtain the following nuclear Gelfand triple 
\begin{equation}\label{eq:nuclear-triple-abstract}
    \mathcal{N} \subset \mathcal{H} \subset \mathcal{N}'.
\end{equation}
In particular, we have $\langle f, \varphi \rangle = (f, \varphi)$ for $f \in \mathcal{H}$, $\varphi \in \mathcal{N}$.
We assume that $\mathcal{N}$ is given as a countable sequence of Hilbert spaces. More precisely, for each $l \in \mathbb{N}$ let $(\mathcal{H}_l,|\cdot|_l)$ be a real separable Hilbert space such that $\mathcal{N} \subset \mathcal{H}_{l+1} \subset \mathcal{H}_l \subset \mathcal{H}$ continuously and the inclusion $\mathcal{H}_{l+1} \subset \mathcal{H}_l$ is a Hilbert Schmidt operator. We assume that the norms $|\cdot|_l$ and $l\in\mathbb{N}$ are increasing; that is, $|\cdot|_l \leq |\cdot|_{l+1}$ on $\mathcal{H}_{l+1}$. The space $\mathcal{N}$ is given as the projective limit of the spaces $(\mathcal{H}_l)_{l \in \mathbb{N}}$, that is, as sets $\mathcal{N} = \bigcap_{l \in \mathbb{N}} \mathcal{H}_l$ and the topology on $\mathcal{N}$ is the coarsest locally convex topology such that all inclusions $\mathcal{N} \subset \mathcal{H}_l$ are continuous. 

This also implies a representation for $\mathcal{N}'$ as an inductive limit. Let $\mathcal{H}_{-l}$ be the dual space of $\mathcal{H}_l$ with respect to $\mathcal{H}$ and let the dual pairing between $\mathcal{H}_{-l}$ and $\mathcal{H}_l$ be indicated by $\langle\cdot,\cdot\rangle$ as well. Then $\mathcal{H}_{-l}$ is a Hilbert space, and its norm is denoted by $\left|  \cdot\right|_{-l}$. It follows from general duality theory that as set $\mathcal{N}'=\bigcup_{l\in\N_0}\mathcal{H}_{-l}$, and $\mathcal{N}'$ is equipped with the inductive topology. That is, the finest locally convex topology such that all inclusions $\mathcal{H}_{-l}\subset\mathcal{N}'$ are continuous.

Thus, we end up with a chain of dense and continuous inclusions:
\[ 
\mathcal{N}\subset\mathcal{H}_{l+1}\subset\mathcal{H}_l\subset\mathcal{H}\subset\mathcal{H}_{-l}\subset\mathcal{H}_{-(l+1)}\subset\mathcal{N}'. 
\]
For each of the real spaces mentioned, we also examine their complex versions, marked with a subscript $\C$, such as the complexification of $\mathcal{H}_l$, which becomes $\mathcal{H}_{l,\C}$, among others. In the following,    we always identify $f = [f_1,f_2]\in\mathcal{H}_{l,\C}$, $f_1,f_2\in\mathcal{H}_l$, $l\in\Z$, with $f=f_1 + \mathrm{i}f_2$. The dual pairing extends as a bilinear form to $\mathcal{N}'_{\mathbb{C}}\times\mathcal{N}_{\mathbb{C}}$.
The space $\mathcal{N}'$ is endowed with the cylinder $\sigma$-algebra $\mathcal{C}_\sigma(\mathcal
{N}')$. 

\begin{definition}[Generalized Fox-$H$ measure] \label{def:FoxHmeasu}
    Under the conditions of Lem\- ma~\ref{lem:HFiniteDensities}, the generalized Fox-$H$ measure (gFHm for short) $\mu_{\Psi}$ is defined as the unique probability measure on $(\mathcal{N}', \mathcal{C}_{\sigma}(\mathcal{N}'))$ such that
    \begin{equation} \label{eq:FoxHmeasure}   
    \int_{\mathcal{N}'}\mathrm{e}^{\mathrm{i}\langle \omega,\varphi \rangle}\,\mathrm{d}\mu_{\Psi}(\omega)=\frac{1}{K}\, _m\!\Psi_p\left[ \genfrac{}{}{0pt}{}{(b_i+\beta_i,\beta_i)_{1,m}}{(a_j+\alpha_j,\alpha_j)_{1,p}} \middle| -\frac{\langle \varphi,\varphi \rangle }{2} \right], \quad \varphi \in \mathcal{N}, 
    \end{equation}
    where $K$ is given in \eqref{eq:norm-constant}.

    We denote this class of measures by $\mathcal{M}_\Psi(\mathcal{N}')$ and the generalized Wright probability space by $(\mathcal{N}',\mathcal{C}_\sigma(\mathcal{N}'),\mu_{\Psi})$. The corresponding $L^p$ Banach spaces of complex-valued $\mathcal{C}_\sigma(\mathcal{N}')$-measurable  functions with integrable $p$-th power are denoted by $L^p(\mu_{\Psi}):=L^p\big(\mathcal{N}',\mathcal{C}_\sigma(\mathcal{N}'),\mu_{\Psi};\C\big)$, $p\geq 1$. The norm in $L^p(\mu_{\Psi})$ is denoted by $\| \cdot\|_{L^p(\mu_{\Psi})}$.
\end{definition} 

The following lemma states the relationship between elements in $\mathcal{M}_\Psi(\mathcal{N}')$ with the finite-dimensional generalized Wright measures $\mu_\Psi^d$ (see Definition~2.11 in \cite{BCDS24}). We refer to Lemma~3.3 in \cite{BCDS24} for a detailed proof.
\begin{lemma}\label{lem:ProjeMoments}
    Let $\varphi_1,\dots,\varphi_d \in \mathcal{N}$ be orthonormal in $\mathcal{H}$, then the image measure of $\mu_{\Psi}$ under the mapping $\mathcal{N}' \ni \omega \mapsto (\langle \omega , \varphi_1 \rangle,\dots, \langle \omega , \varphi_d \rangle)\in\R^d$ is the finite-dimensional gWm $\mu_{\Psi}^d$.
\end{lemma}

\begin{definition}[Mixed moments]\label{def:moments}
Let $n\in\mathbb{N}$, $\mu_{\Psi}\in\mathcal{M}_\Psi(\mathcal{N}')$, and $\varphi_i\in \mathcal{N}$, $i=1,\dots,n$, be given. The generalized moments of $\mu_{\Psi}$ are defined by
\[
 M_n^{\mu_{\Psi}}(\varphi_1,\dots, \varphi_n):=\int_{\mathcal{N}'}\langle \omega^{\otimes n}, \varphi_1\otimes\dots\otimes \varphi_n \rangle\, \mathrm{d}\mu_{\Psi}(\omega).
\]
\end{definition}
The result of Lemma~\ref{lem:ProjeMoments} implies directly the following. 

\begin{theorem}\label{thm:MomentmuHInfiniteDimension}
 Let the assumptions of Lemma \ref{lem:HFiniteDensities} hold, $\varphi \in \mathcal{N}$, and $n \in \N_0$ be given. Then the odd moments $M_{2n+1}^{\mu_{\Psi}}$ are zero and the even moments $M_{2n}^{\mu_{\Psi}}$ are given by
\[   
M_{2n}^{\mu_{\Psi}}(\varphi):=M_{2n}^{\mu_{\Psi}}(\varphi,\dots,\varphi)=\frac{1}{K}\frac{\prod_{i=1}^m \Gamma(b_i + \beta_i (n+1))}{\prod_{i=1}^p \Gamma(a_i + \alpha_i (n+1))} \frac{(2n)! }{n!2^{n}}  \langle\varphi,\varphi\rangle^{n}.
\]
In particular, for $\varphi,\psi \in \mathcal{N}$ we obtain
\begin{eqnarray}
\|\langle\cdot,\varphi\rangle\langle\cdot,\psi\rangle\|_{L^1(\mu_\Psi) }=M_{2}^{\mu_{\Psi}}(\varphi,\psi)&=&\frac{1}{K}\frac{\prod_{i=1}^m \Gamma(b_i + 2\beta_i) }{\prod_{i=1}^p \Gamma(a_i + 2\alpha_i) } \langle \varphi, \psi \rangle,\label{eq:mixed-moment2}\\
\|\langle\cdot,\varphi\rangle\|^2_{L^2(\mu_\Psi) }=M_{2}^{\mu_{\Psi}}(\varphi,\varphi)&=&\frac{1}{K}\frac{\prod_{i=1}^m \Gamma(b_i + 2\beta_i) }{\prod_{i=1}^p \Gamma(a_i + 2\alpha_i) }|\varphi|^2_0.\label{eq:mixed-moment2a}
\end{eqnarray}
\end{theorem}

\begin{remark}\label{rem:ext-dual-pairing}
Using the above theorem, it is possible to extend the dual pairing to $\mathcal{N}'\times\mathcal{H}$. More precisely, given $f\in\mathcal{H}$ there exists a sequence $(\varphi_n)_{n \in \N}\subset\mathcal{N}$ such that $\varphi_n\longrightarrow f$, $n\to\infty$ in $\mathcal{H}$.   It follows from \eqref{eq:mixed-moment2a} that $(\langle\cdot,\varphi_n\rangle)_{n\in\mathbb{N}}$ is a Cauchy sequence in $L^2({\mu_\Psi})$, hence it converges. Choosing a subsequence $(\langle \cdot, \varphi_{n_k} \rangle)_{k \in \N}$, we then define $\langle\cdot,f\rangle$ as an $L^2(\mu_{\Psi})$-limit of $\langle \cdot, \varphi_{n_k} \rangle$, that is
\begin{equation}
\langle\omega,f\rangle:=\lim_{k\to\infty}\langle\omega,\varphi_{n_k}\rangle,\quad \mu_\Psi\mbox{-a.a.}\;\omega\in\mathcal{N}'.
\end{equation}
\end{remark} 

There is a standard way to construct test function spaces in non-Gaussian analysis; see e.g.~\cite{KSWY98}. Given the applications we have in mind, see Section~\ref{sec:applications} below, we use the Kondratiev test function space $(\mathcal{N})^{1}$, and its dual space with respect to $L^2(\mu_\Psi)$, i.e., $(\mathcal{N})^{-1}_{\mu_{\Psi}}$, such that we obtain the triple:
\[  (\mathcal{N})^{1}\subset L^2(\mu_{\Psi}) \subset(\mathcal{N})^{-1}_{\mu_{\Psi}}.\]
The dual pairing $\langle\!\langle \cdot,\cdot \rangle\!\rangle_{\mu_{\Psi}}$ between $ (\mathcal{N})^{-1}_{\mu_{\Psi}} $ and $(\mathcal{N})^{1}$ corresponds to the bilinear extension of the inner product from $L^2(\mu_{\Psi})$.\\
To define the $S_{\mu_{\Psi}}$-transform, we first need to introduce the normalized exponential $e_{\mu_{\Psi}}(\cdot, \xi)$, $\xi\in\mathcal{N}_{\mathbb{C}}$:
 \[
 \mathrm{e}_{\mu_{\Psi}}(\cdot,\xi):= \frac{\e^{\langle \cdot, \xi\rangle}}{l_{\mu_{\Psi}}(\xi)}:=\frac{\e^{\langle \cdot, \xi\rangle}}{\mathbb{E}_{\mu_\Psi}[\mathrm{e}^{\langle\cdot,\xi\rangle}]}.
 \]
The normalized exponential $e_{\mu_{\Psi}}(\cdot, \xi)$ is a test function in $(\mathcal{N})^{1}$ of finite order (cf.~Example 6 in \cite{KSWY98}) if and only if $\xi\in\mathcal{N}_\mathbb{C}$ is in a neighborhood of zero, that is, for any $\xi\in U_{l,k}=\{\xi \in \mathcal{N}_{\mathbb{C}}\mid \, 2^k |\xi|_l <1\}$.  
The $S_{\mu_{\Psi}}$-transform of $\Phi \in (\mathcal{N})^{-1}_{\mu_{\Psi}}$ is defined by
\[ (S_{\mu_{\Psi}}\Phi )(\xi):= \langle\!\langle \Phi, e_{\mu_{\Psi}}(\cdot,\xi) \rangle\!\rangle_{\mu_{\Psi}}=\frac{1}{l_{\mu_{\Psi}}(\xi)}\int_{\mathcal{S}'}\e^{\langle \omega, \xi\rangle}\Phi(\omega)\,\mu_{\Psi}(\mathrm{d}\omega), \quad \xi \in  U_{l,k}, \]
where $l_{\mu_{\Psi}}(\xi)$ follows from \eqref{eq:FoxHmeasure} at $\varphi=-\mathrm{i}\,\xi$:
\[ 
l_{\mu_{\Psi}}(\xi)=\frac{1}{K}\, _m\!\Psi_p\left[ \genfrac{}{}{0pt}{}{(b_i+\beta_i,\beta_i)_{1,m}}{(a_j+\alpha_j,\alpha_j)_{1,p}} \middle|\, \frac{\langle \xi,\xi \rangle }{2} \right].
\]
The $S_{\mu_{\Psi}}$-transform characterizes the elements of $(\mathcal{N})^{-1}_{\mu_\Psi}$ in terms of holomorphic functions.
\begin{theorem}[Cf.~Thm.~8.34 in \cite{KSWY98}]\label{thm:StransIsomor}
    The $S_{\mu_{\Psi}}$-transform is a topological isomorphism from $(\mathcal{N})^{-1}_{\mu_{\Psi}}$ to $\mathrm{Hol}_{0}(\mathcal{N}_{\mathbb{C}})$.
\end{theorem}

The above characterization theorem leads directly to two corollaries for integrals of $(\mathcal{N})^{-1}_{\mu_{\Psi}}$-valued elements (in the weak sense) and the convergence of sequences in $(\mathcal{N})^{-1}_{\mu_{\Psi}}$.

\begin{theorem}[See Thm.~4.7~in \cite{BCDS24}]\label{thm:CharIntegrableMap}
    Let $(T,\mathcal{B},\nu)$ be a measure space and $\Phi_t\in (\mathcal{N})^{-1}_{\mu_{\Psi}}$ for all $t\in T$. Let $\mathcal{U}\subset\mathcal{N}_\C$ be an appropriate neighborhood of zero and $C$ a positive constant such that:
    \begin{description}
	\item{1.} $(S_{\mu_{\Psi}}\Phi_\cdot)(\xi)\colon T\to\C$ is measurable for all $\xi\in\mathcal{U}$.
	\item{2.} $\int_T |(S_{\mu_{\Psi}}\Phi_t)(\xi)|\,\mathrm{d}\nu(t) \leq C$ for all $\xi\in\mathcal{U}$.
    \end{description}
    Then there exists $\Xi\in (\mathcal{N})^{-1}_{\mu_{\Psi}}$ such that for all $\xi\in\mathcal{U}$
    \[
    (S_{\mu_{\Psi}}\Xi)(\xi) = \int_T (S_{\mu_{\Psi}}\Phi_t)(\xi)\,\mathrm{d}\nu(t).
    \]
    We denote $\Xi$ by $\int_T \Phi_t\,\mathrm{d}\nu(t)$ and call it the weak integral of the family $\{\Phi_t, t \in T\}$.
\end{theorem}

\begin{theorem}[See Thm.~2.12~in \cite{JahnII}]\label{thm:StransConverg}
Let $\{\Phi_n\}_{n\in \mathbb{N}}$ be a sequence in $(\mathcal{N})_{\mu_{\Psi}}^{-1}$. Then $\{\Phi_n\}_{n \in \mathbb{N}}$ converges strongly in $(\mathcal{N})_{\mu_{\Psi}}^{-1}$ if and only if there exists $l,k \in \mathbb{N}$ with the following two properties:
\begin{itemize}
\item[i)] $\{(S_{\mu_{\Psi}}\Phi_n) (\xi)\}_{n \in \mathbb{N}}$ is a Cauchy sequence for all $\xi \in U_{l,k}$;
\item[ii)] $(S_{\mu_{\Psi}}\Phi_n)(\cdot)$ is holomorphic on $U_{l,k}$ and there is a constant $C>0$ such that
\[ |(S_{\mu_{\Psi}}\Phi_n) (\xi)|\leq C  \]
for all $\xi \in U_{l,k}$ and for all $n \in \mathbb{N}$.
\end{itemize} 
\end{theorem}

\begin{definition}
For $\Phi \in (\mathcal{N})^{-1}_{\mu_{\Psi}}$ and $\xi \in U_{l,k}$, we define the $T_{\mu_{\Psi}}$-transform by
\[  (T_{\mu_{\Psi}}\Phi)(\xi)=\langle\! \langle \Phi, \e^{\mathrm{i}\langle \cdot, \xi \rangle } \rangle \!\rangle_{\mu_{\Psi}}=\int_{\mathcal{N}'}\e^{\mathrm{i}\langle \omega, \xi\rangle}\Phi(\omega)\, \mathrm{d}\mu_{\Psi}(\omega). \]
\end{definition}
\begin{remark}
\begin{enumerate}
    \item For $\xi=0$, we have $\exp{(\mathrm{i} \langle \cdot, \xi \rangle)}=1$, which implies a relation between $T_{\mu_\Psi}$-transform and the generalized expectation of $\Phi$ 
    \[ 
    (T_{\mu_{\Psi}}\Phi)(0)=\mathbb{E}_{\mu_{\Psi}}(\Phi). 
    \]
    \item As an application, we show that Donsker's delta $\delta(\langle \cdot, f \rangle )$, $f \in \mathcal{H}$, as an element in $(\mathcal{N})^{-1}_{\mu_{\Psi}}$, see Theorem~5.2 in \cite{BCDS24}.
\end{enumerate}
\end{remark}

\subsection{The Radon-Nikodym Derivative in Euclidean Spaces}
In the following, we recall that the Radon-Nikodym derivative of a generalized Wright measure is expressed by a Fox-$H$ density with all moments when the nuclear triple \eqref{eq:nuclear-triple-abstract} is realized by the $d$-dimensional Euclidean space. That is, $\R^d\subset\R^d\subset\R^d$, $d \in \N$.
\begin{theorem}[See Thm~2.14~in \cite{BCDS24}]\label{thm:FiniteDimensionalFoxHGaussianDensity}
    For $d \in \N$, let the assumptions of Lemma~\ref{lem:HFiniteDensities} hold together with $2(b_j + \beta_j)> \beta_jd$, for $j=1,\dots,m$.
    Then the $d$-dimensional generalized Wright measure $\mu_{\Psi}^d$ is absolutely continuous w.r.t.~the $d$-dimensional Lebesgue measure and its Radon-Nikodym density is expressed, for every $x \in \R^d$, by
    \begin{equation}\label{eq:FiniteDimensionalFoxHGaussianDensity}
    \varrho_H^d(x):= \frac{1}{(2 \pi)^{d/2}K} H^{m+1,0}_{p,m+1}\left[ \frac{(x,x)}{2}\,\bigg| \genfrac{}{}{0pt}{}{(a_i+\alpha_i(1-d/2),\alpha_i)_p}{(0,1),(b_{j-1}+\beta_{j-1}(1-d/2),\beta_{j-1})_{2,m+1} }\right]. 
    \end{equation}
\end{theorem}

\begin{remark}We would like to emphasize that the density $\varrho_H^d(\cdot)$ in Equation~\eqref{eq:FiniteDimensionalFoxHGaussianDensity} coincides with Equation~(3.23) in \cite{S92} for $b_1=0,\, \beta_1=1,\, a_1=1-\rho$, and $ \alpha_1=\rho$. In this particular case, $\varrho_H^d$ coincides with the density of the $d$-dimensional Mittag-Leffler measure.
\end{remark}

Measures in the class $\mathcal{M}_\Psi(\mathbb{R}^d)$ can be expressed as a mixture of Gaussian measures with a proper mixing family of probability measures on $(0,\infty)$. More precisely, we have

\begin{property}\label{property:mixture-fd}
Let the hypothesis of Theorem~\ref{thm:FiniteDimensionalFoxHGaussianDensity} hold. 
\begin{enumerate}
    \item The density $\varrho_H^d(\cdot)$ is a Gaussian mixture as follows:
    \begin{align*}
      \varrho_H^d(x)&=\frac{1}{(2 \pi)^{d/2}K} H^{m+1,0}_{p,m+1}\left(\frac{(x,x)}{2}\right)\\
      &=\frac{1}{(2 \pi)^{d/2}K}\int_0^\infty  \frac{1}{\tau^{d/2}} \exp\left(-\frac{(x,x)}{2 \tau} \right)  H^{m,0}_{p,m}(\tau)\,\mathrm{d}\tau, \quad x \in \R^d.
      \end{align*}
      \item Equivalently, the measure $\mu_\Psi^d$ is given by the mixture of Gaussian measures
      \[
      \mu_\Psi^d=\int_0^\infty\gamma_s^d\,\mathrm{d}\nu_H(s),
      \] 
      where $\gamma_s^d$, $s>0$, is the Gaussian measure on $\R^d$ with variance $sI$ with $I$ the $d \times d$-identity matrix and $\nu_H$ is the probability measure on $(0,\infty)$ with density $\varrho_H$ from Lemma~\ref{lem:HFiniteDensities}.
      \item The measures $\mu^d_\Psi$, $d\in \N$, are multivariate elliptical distributions with density generator
\[ 
 H^{m+1,0}_{p,m+1}(x)=\int_0^\infty  \frac{1}{\tau^{d/2}} \exp\left(-\frac{(x,x)}{2 \tau} \right)  H^{m,0}_{p,m}(\tau)\,\mathrm{d}\tau, \quad x \in \R^d, 
\] see Remark 2.17 in \cite{BCDS24}.
      \item The generalized Wright measures $\mu_\Psi$ are the mixture of Gaussian measures
      \[
      \mu_\Psi=\int_0^\infty \mu^{(s)}\,\mathrm{d}\nu_H(s),
      \]
      where $\mu^{(s)}$ is the Gaussian measure on $\mathcal{N}'$ with variance $s>0$. 
\end{enumerate}

\end{property}

\begin{remark}
The mixture property of the family of measures shown above has important consequences:
\begin{itemize}
    \item[1.] The class of processes associated with this class of measures consists of time-changed Gaussian processes;
    \item[2.] According to the general theory of Gaussian measures (refer to Corollary~4.4 in \cite{BCDS24}), this class cannot be ergodic, and as a result, neither can the associated stochastic process.
\end{itemize}
\end{remark}

\section{Generalized Fox-\texorpdfstring{$H$}{H} Process} \label{Sec 3: Gene Stoch Proce}
In this section, we introduce generalized stochastic processes as continuous linear mappings on the space $\mathcal{N}$ and $\mathcal{H}$ and study their key properties. If we choose a particular nuclear triple (e.g., the triple of Schwartz test functions and generalized functions) the generalized stochastic processes may be realized as stochastic processes at time $t \ge 0$. To define the latter, we briefly describe the stochastic counterpart of this infinite-dimensional analytical framework.

\subsection{Generalized Stochastic Processes}
In Section~\ref{subsec:gFHm} we have introduced the generalized Wright measures on the measurable space $(\mathcal{N}', \mathcal{C}_{\sigma}(\mathcal{N}'))$ by exploiting Bochner-Minlos' theorem on an abstract nuclear space, namely $\mathcal{N} \subset \mathcal{H} \subset \mathcal{N}'$. Consequently, we can interpret the dual pairing between $\mathcal{N}'$ and $\mathcal{N}$ as a generalized random variable. We denote them by  
\[
\mathcal{N}\ni \varphi\mapsto X(\varphi):=\langle \cdot, \varphi \rangle \in \mathbb{R}. 
\]
It follows from Theorem~\ref{thm:MomentmuHInfiniteDimension} that each generalized random variable $X(\varphi)$ is an element in $L^2(\mu_{\Psi})$. In addition, we can extend $X(\cdot)$ from $\mathcal{N}$ to $\mathcal{H}$ via an approximation procedure through the density of $\mathcal{N}$ in $\mathcal{H}$, see Remark~\ref{rem:ext-dual-pairing}. Thus, for  $f\in\mathcal{H}$ we get a generalized random variable
\[
X(f): (\mathcal{N}', \mathcal{C}_{\sigma}(\mathcal{N}')) \to (\mathbb{R},\mathcal{B}(\mathbb{R})), \, \omega \mapsto\langle \omega, f \rangle, \quad \mu_\Psi\mbox{-a.a.}\;\omega\in\mathcal{N}', 
\]
where $\mathcal{B}(\mathbb{R})$ is the Borel $\sigma$-algebra on $\mathbb{R}$. They share the same properties as those $X(\varphi)$, $\varphi \in \mathcal{N}$. 

To interpret the above generalized random variables as a stochastic process at time $t \geq 0$, we are going to specify the above-cited abstract nuclear triple to be 
\[   \mathcal{S}(\mathbb{R}) \subset L^2(\mathbb{R}) \subset \mathcal{S}(\mathbb{R})', \]
and choose $f \in L^2(\mathbb{R})$ to be a specific function depending on $t$. More specifically, for $t>0$ and $\mathbbm{H} \in (0,1)$, we choose $f(x)=M_{-}^\mathbbm{H} \mathbbm{1}_{[0,t)}(x)$, $x \in \mathbb{R}$, where $M^\mathbbm{H}_{-}$ stands for the Riemann-Liouville fractional derivative (if $\mathbbm{H} \in (0,1/2)$) or the fractional integral (if $\mathbbm{H} \in (1/2,1)$). See the definition in Equation~\eqref{eq:FractionalOperators} for details.

The latter choice allows us to define a stochastic process (see Definition~\ref{def:GenWrighProcess}) and study its properties.

\subsection{Riemann-Liouville Fractional Operators}
In this subsection, we recall the definitions of right-sided and left-sided Riemann-Liouville fractional derivatives, $D_{\pm}^{\sigma}$, and integrals, $ I^{\sigma}_{\pm}$, needed to define the main object of this section. The interested reader can find more details in classical books \cite{SKM93} and \cite{KilSvri}.  

\begin{definition}[cf.~Eq.~(2.3.1)-(2.3.2) in \cite{KilSvri}]\label{def:leftrightRLIntegrals}
    Let $ \sigma \in (0,1)$ and $f \in L^1(\R)$ be given. The left-sided Riemann-Liouville fractional integral of order $\sigma$, denoted by $I^\sigma_{+}$, is defined as follows:

    \begin{equation}\label{eq:leftsidedRLIntegral}
	\left( I^\sigma_{+}f\right)(x) := \frac{1}{\Gamma(\sigma)} \int_{-\infty}^x f(t) (x-t)^{\sigma-1}\,\mathrm{d} t,\quad x\in\R, 
 \end{equation}
 while the right-sided Riemann-Liouville fractional integral of order $\sigma$, $I^{\sigma}_{-}$, as
 \begin{equation}\label{eq:rightsidedRLIntegral}
	\left( I^\sigma_{-}f\right)(x) := \frac{1}{\Gamma(\sigma)} \int_x^\infty f(t) (t-x)^{\sigma-1}\,\mathrm{d} t,\quad x\in\R.
\end{equation}
\end{definition}

\begin{definition}[cf.~Eq.~(2.3.6)-(2.3.7) in \cite{KilSvri}]\label{def:leftrightRrivatives}

Let $ \sigma \in (0,1)$ and $f \in L^1(\R)$  be given. The left-sided Riemann-Liouville fractional derivative of order $\sigma$, denoted by $D^\sigma_{+}$, is defined as follows:

    \begin{equation}\label{eq:leftsidedRivative}
	\left( D^\sigma_{+}f\right)(x) = \frac{1}{\Gamma(1-\sigma)} \frac{\mathrm{d}}{\mathrm{d} x} \int_{-\infty}^x f(t) (x-t)^{-\sigma}\,\mathrm{d} t,\quad x\in\R, 
 \end{equation}
 while the the right-sided Riemann-Liouville fractional derivative of order $\sigma$, $D^{\sigma}_{-}$, as
 \begin{equation}\label{eq:rightsidedRLDerivative}
	\left( D^\sigma_{-}f\right)(x) = -\frac{1}{\Gamma(1-\sigma)} \frac{\mathrm{d}}{\mathrm{d} x} \int_x^\infty f(t) (t-x)^{-\sigma}\,\mathrm{d} t,\quad x\in\R. 
\end{equation}
\end{definition}
In the following, we adopt the convention $t_+^\mathfrak{a}:=  t^\mathfrak{a}$, if $t>0$ and zero otherwise, to express the fractional operators on indicator functions. For $\sigma \in (0,1)$, we have
\begin{equation}
	\left( I^\sigma_{-}\mathds{1}_{(a,b)}\right) (x) = \frac{1}{\Gamma(\sigma+1)}\left( (b-x)_{+}^\sigma  - (a-x)^\sigma_{+} \right).
\end{equation}
and 
\[
	\left( D^\sigma_{-}\mathds{1}_{(a,b)}\right) (x) = \frac{1}{\Gamma(1-\sigma)}\left( (b-x)_{+}^{-\sigma} - (a-x)^{-\sigma}_{+} \right).
\]
For $\mathbbm{H} \in (0,1)$, we define the operator $M^{\mathbbm{H}}_{\pm}$ by	
\begin{equation}\label{eq:FractionalOperators}
	M^{\mathbbm{H}}_{\pm} f := \begin{cases}
		K_\mathbbm{H} D^{-(\mathbbm{H}-1/2)}_{\pm}f, & \mathbbm{H}\in (0,1/2), \\
		f, & \mathbbm{H}=1/2, \\
		K_\mathbbm{H} I^{\mathbbm{H}-1/2}_{\pm}f, & \mathbbm{H}\in (1/2,1), 
	\end{cases} 
\end{equation}
where the constant is
\[	
K_\mathbbm{H}:=\Gamma\left( \mathbbm{H}+1/2 \right)\left( \int_0^\infty \left((1+s)^{\mathbbm{H}-1/2}-s^{\mathbbm{H}-1/2} \right)^2\,\mathrm{d} s+\frac{1}{2\mathbbm{H}}\right)^{-1/2}.
\]
We point out the following facts in order to consistently define the stochastic processes and study their properties.
\begin{lemma}[]
Let $\mathbbm{H} \in (0,1)$, then the following properties hold:

\begin{enumerate}
    \item for $a,b \in \mathbb{R}$, $M^\mathbbm{H}_\pm \mathds{1}_{[a,b)}\in L^2(\R,\mathrm{d} x);$
    \item for $t_1, t_2 \geq 0$ we have 
    \begin{equation}\label{eq: fBm Covariance}
        (M^{\mathbbm{H}}_{-} \mathds{1}_{[0,t_1)}, M^{\mathbbm{H}}_{-} \mathds{1}_{[0,t_2)})_{L^2(\R,\mathrm{d}x)}=\frac{1}{2}(t_1^{2\mathbbm{H}} + t_2^{2\mathbbm{H}} -|t_1-t_2|^{2\mathbbm{H}}).
    \end{equation}
\end{enumerate}
\end{lemma}
We refer to \cite{Mishura08} for the proof and further details.

\subsection{The Generalized Fox-\texorpdfstring{$H$}{H} Process. Definition and Properties}
In this subsequent, we give the definition of generalized Fox-$H$ processes and study their moments, covariance structures, their increments, and their characteristic function. The notation given to define the gFHp allows to explicit the dependence on the parameters for the representation as a fBm time-changed later in Proposition~\ref{Prop:SubordRepresen}, defining $A=((a_i,\alpha_i))_{1,p}$ and $B=((b_j,\beta_j))_{1,m}$.
\begin{definition}\label{def:GenWrighProcess}
Let  the assumptions of Lemma \ref{lem:HFiniteDensities} hold, and $(\mathcal{S}', \mathcal{C}_{\sigma}(\mathcal{S}'), \mu_{\Psi})$ be the related generalized Wright probability space.\\ For $\mathbbm{H} \in (0,1)$, we define the generalized Fox-$H$ process as the collection $\{ X^{\mathbbm{H}, A, B}_t,t \geq 0 \}$, where
\begin{eqnarray*}X^{\mathbbm{H}, A, B}_t: (\mathcal{S}', \mathcal{C}_{\sigma}(\mathcal{S}')) &\to& (\mathbb{R}, \mathcal{B}(\mathbb{R}))\\
\omega &\mapsto& X^{\mathbbm{H}, A, B}_t(\omega):= \langle\omega,M^{\mathbbm{H}}_{-}\mathds{1}_{[0,t)}\rangle \quad  \mu_{\Psi}\mbox{-}\text{a.a.} \;  \omega \in \mathcal{S}' \end{eqnarray*}
where $M^{\mathbbm{H}}_{-}$ is the operator defined in Equation~\eqref{eq:FractionalOperators}.
\end{definition}

We can characterize the generalized Fox-$H$ processes by employing Theorem~\ref{thm:MomentmuHInfiniteDimension} in conjunction with Theorem~\ref{thm:FiniteDimensionalFoxHGaussianDensity} as follows.

\begin{proposition}\label{prop:CFDensityMomentsCovariance}
    Let the assumptions of Lemma~\ref{lem:HFiniteDensities} hold, $\mathbbm{H} \in (0,1)$ be given and $K$ be as in Equation~\eqref{eq:norm-constant}. Then the generalized Fox-$H$ processes have the following properties:
    \begin{enumerate}
        \item For $n \in \mathbb{N}$ and $0<t_1<\dots<t_n$, the characteristic function of $X^{\mathbbm{H},A,B}:=(X^{\mathbbm{H},A,B}_{t_1}
        ,\dots,X^{\mathbbm{H},A,B}_{t_n})$ equals 
        \[  
        \mathds{E}\big[\mathrm{e}^{\mathrm{i}(\lambda, X^{\mathbbm{H}, A, B}) }\big]=\frac{1}{K}\, _m\!\Psi_p\left[ \genfrac{}{}{0pt}{}{(b_i+\beta_i,\beta_i)_{1,m}}{(a_j+\alpha_j,\alpha_j)_{1,p}} \middle| -\frac{ (\lambda, \Sigma_{\mathbbm{H},n} \lambda) }{2} \right], \quad \lambda \in \mathbb{R}^n.
        \]
        where the matrix $\Sigma_{\mathbbm{H},n}=(\sigma_{l,k})^n_{l,k=1}$ is determined by 
        \[
        \sigma_{l,k}=\frac{1}{2}\bigl(|t_l|^{2\mathbbm{H}}+|t_k|^{2\mathbbm{H}}-|t_l-t_k|^{2\mathbbm{H}}\bigr).
        \]
        \item Let $\varrho_{B^{\mathbbm{H}}}(x)$, $x \in \mathbb{R}^n$, be the joint probability density function of fractional Brownian motion $B^\mathbbm{H}=(B^\mathbbm{H}_{t_1}
        ,\dots,B^\mathbbm{H}_{t_n})$ with Hurst parameter $\mathbbm{H}$, that is,
        \[ 
        \varrho_{B^{\mathbbm{H}}}(x)=\frac{1}{(2 \pi)^{n/2} \det(\Sigma_{\mathbbm{H},n})^{1/2}} \exp\left(-\frac{(x,\Sigma_{\mathbbm{H},n}^{-1}x)}{2} \right).  
        \]
        Then the joint probability density function of $X^{\mathbbm{H},A,B}=(X^{\mathbbm{H},A,B}_{t_1}
        ,\dots,X^{\mathbbm{H},A,B}_{t_n})$ is given by
        \begin{equation}\label{eq:WrighProcessDensity}
           \varrho_{X^{\mathbbm{H}, A, B}}(x)=\frac{1}{K}\int_0^\infty  \frac{1}{\tau^{n/2}} \varrho_{B^{\mathbbm{H}}}\left(\frac{x}{\sqrt{\tau}}\right)  H^{m,0}_{p,m}(\tau)\,\mathrm{d}\tau, \quad x \in \mathbb{R}^n.
        \end{equation}
        \item For each $t\geq0$, the moments of any order of the generalized Fox-$H$ process are given by 
\[
\begin{cases}
\mathbb{E}\big[(X^{\mathbbm{H},A,B}_
t)^{2n+1}\big] & =0,\\[.2cm]
\mathbb{E}\big[(X^{\mathbbm{H},A,B}_t)^{2n}\big] & =\frac{1}{K}\frac{\prod_{i=1}^m \Gamma(b_i + \beta_i (n+1))}{\prod_{i=1}^p \Gamma(a_i + \alpha_i (n+1))}t^{2n\mathbbm{H}}.
\end{cases}
\]
\item The covariance function has the form 
\begin{equation*}
\mathbb{E}\big[X^{\mathbbm{H},A,B}_t X^{\mathbbm{H},A,B}_s\big]=\frac{1}{2K}\frac{\prod_{i=1}^m \Gamma(b_i + 2\beta_i )}{\prod_{i=1}^p \Gamma(a_i + 2\alpha_i) }\big(t^{2\mathbbm{H}}+s^{2\mathbbm{H}}-|t-s|^{2\mathbbm{H}}\big),\quad t,s\geq0.
\end{equation*}
\end{enumerate}
\end{proposition}
\begin{proof}
    \begin{enumerate}
        \item Let $n \in \mathbb{N}$, $0<t_1<\dots<t_n$ and $\lambda \in \mathbb{R}^n$. Define $f(\cdot):=\sum_{i=1}^n \lambda_i M^{\mathbbm{H}}_{-}\mathds{1}_{[0,t_i)}(\cdot) \in L^2(\mathrm{d}x)$, we have \[(\lambda, X^{\mathbbm{H},A,B})= \sum_{i=1}^n \lambda_i X^{\mathbbm{H},A,B}_{t_i}=\sum_{i=1}^n \lambda_i\langle\cdot,   M^{\mathbbm{H}}_{-}\mathds{1}_{[0,t_i)}\rangle=\langle\cdot, f\rangle.\]  Hence,
        \begin{eqnarray*}
        \mathds{E}\big(\mathrm{e}^{\mathrm{i}(\lambda, X^{\mathbbm{H}, A, B}) }\big)&=&\mathds{E}_{\mu_{\Psi}}\big(\mathrm{e}^{\mathrm{i}\langle\omega, f\rangle} \big) \\
        &=& \frac{1}{K}\, _m\!\Psi_p\left[ \genfrac{}{}{0pt}{}{(b_i+\beta_i,\beta_i)_{1,m}}{(a_j+\alpha_j,\alpha_j)_{1,p}} \middle| -\frac{ \langle f, f \rangle }{2} \right].
        \end{eqnarray*}
    An easy computation show that $\langle f, f \rangle = (\lambda, \Sigma_{\mathbbm{H},n} \lambda) $, we obtain the claim. 
    \item We obtain the density of generalized Fox-$H$ processes by considering the LT representation of $_m\Psi_p(-\cdot)$ in Equation~\eqref{eq:LTdensitiesFiniteMoments} and the characteristic function given in point 1 above.
    \item The result follows by applying Theorem~\ref{thm:MomentmuHInfiniteDimension} on a subsequence converging to $f(\cdot)=M^{\mathbbm{H}}_{-}\mathds{1}_{[0,t)}(\cdot) \in L^2(\mathrm{d}x)$, see also Remark~\ref{rem:ext-dual-pairing}. 
    \item We get the claimed covariance by considering the point 1 above, Theorem~\ref{thm:MomentmuHInfiniteDimension} and Remark~\ref{rem:ext-dual-pairing} as in the above point 3.\qedhere
    \end{enumerate}
\end{proof}

\begin{proposition}
    Let the assumptions of Lemma~\ref{lem:HFiniteDensities} hold, $\mathbbm{H} \in (0,1)$ be given and $K$ be as in Equation~\eqref{eq:norm-constant}. Then 
    \begin{enumerate}
        \item For each $t,s\geq0$, the characteristic function of the increments is 
                \begin{equation*}
                    \mathbb{E}\big[e^{i\lambda(X^{\mathbbm{H},A,B}_t-X^{\mathbbm{H},A,B}_s)}\big]=\frac{1}{K}\,_m\Psi_p\left(-\frac{\lambda^{2}}{2}|t-s|^{2\mathbbm{H}}\right),\quad\lambda\in\mathbb{R}.
                \end{equation*}
        \item The process $X^{\mathbbm{H},A,B}_t$ is $\mathbbm{H}$-self-similar with stationary increments. 
        \item For all $n \in \mathbb{N}$ there exists $C_1>0$ such that
        \begin{equation}\label{eq:ContinuousVersion}  \mathbb{E}_{\mu_{\Psi}} \left(|X^{\mathbbm{H},A,B}_t-X^{\mathbbm{H},A,B}_s|^{2n} \right)= C_1 |t-s|^{n 2\mathbbm{H}}. \end{equation}
        Hence, there exists a version of the process with $\gamma $-H\"{o}lder continuous sample paths, for $0<\gamma <\mathbbm{H}-1/2n$, provided that $n \in \mathbb{N}$ satisfies $2n\mathbbm{H}> 1$.
    \end{enumerate}
\end{proposition}
\begin{proof}
    \begin{enumerate}
        \item It is straightforward by considering that, for $0<s<t$, $X^{\mathbbm{H},A,B}_t-X^{\mathbbm{H},A,B}_s=\langle \cdot, M^{\mathbbm{H}}_{-} \mathds{1}_{[0,t)} \rangle - \langle \cdot, M^{\mathbbm{H}}_{-} \mathds{1}_{[0,s)} \rangle=\langle \cdot, M^{\mathbbm{H}}_{-} \mathds{1}_{[s,t)} \rangle$ and $|M^{\mathbbm{H}}_{-} \mathds{1}_{[s,t)} |^2_0=|t-s|^{2\mathbbm{H}}$, by Equation~\eqref{eq: fBm Covariance}.
        \item The generalized Fox-$H$ processes are self-similar with  parameter $\mathbbm{H}$ since  $t^{\mathbbm{H}}X^{\mathbbm{H},A,B}_1$ and $X^{\mathbbm{H},A,B}_t$ have the same characteristic function. Similarly, $X^{\mathbbm{H},A,B}_{t-s}$ and $X^{\mathbbm{H},A,B}_t -X^{\mathbbm{H},A,B}_s$ share the same characteristic function, considering that $|M^{\mathbbm{H}}_{-} \mathds{1}_{[s,t)} |^2_0=| M^{\mathbbm{H}}_{-} \mathds{1}_{[0,t-s)} |^2_0$, see Equation~\eqref{eq: fBm Covariance}.
        \item In order to state Equation~\eqref{eq:ContinuousVersion} is enough considering the even moments in Proposition~\ref{prop:CFDensityMomentsCovariance}-3 and the Kolmogorov continuity theorem.\qedhere
    \end{enumerate}\end{proof}

\begin{remark}\label{rem:Cases}
The class of generalized Fox-$H$ processes cover a wide variety of well-known processes. Here we give some examples:
\begin{enumerate}
    \item We obtain Brownian motion for $p=0,m=0$ and $\mathbbm{H}=1/2$, the characteristic function of such gFHp reduce to 
    \[  \mathbb{E}[\mathrm{e}^{\mathrm{i}(\lambda,X^{\mathbbm{H}, A, B})}]= \, _0 \Psi_0\left(-\frac{\langle f,f \rangle}{2}\right)\overset{*}{=}\mathrm{e}^{-\frac{\sum_{j,i=1}^n \lambda_i  \lambda_j (t_j \wedge t_i)}{2}} \]
    where in $*$ we used (C0) of Corollary 9 of \cite{BCDS23} and $f$ defined as in the proof of Proposition~\ref{prop:CFDensityMomentsCovariance}.
    \item Fractional Brownian motion with Hurst parameter $\mathbbm{H} \in (0,1)$ corres\-ponds to the generalized Fox-$H$ processes for $p=0,m=0$;

    \item The grey Brownian motion $B^{\beta}, \beta \in (0,1)$ is given by choosing $m=p=1$ and $(b_1,\beta_1)=(0,1)$, $(a_1,\alpha_1)=(1-\beta, \beta)$ and $\mathbbm{H}=1/2$;

    \item We find the generalized grey Brownian motion (hereinafter ggBm) $B^{\alpha, \beta}, \alpha \in (0,2), \beta \in (0,1)$ is given by choosing $m=p=1$ and $(b_1,\beta_1)=(0,1)$, $(a_1,\alpha_1)=(1-\beta, \beta)$ and $\mathbbm{H}=\alpha/2$, thus 
    \begin{eqnarray*} \mathbb{E}[\mathrm{e}^{\mathrm{i}(\lambda,X^{\mathbbm{H}, A, B})}]&=&\,\frac{1}{K}\, _1\!\Psi_1\left[ \genfrac{}{}{0pt}{}{(1,1)}{(1,\beta)} \middle| -\frac{ \langle f, f \rangle }{2} \right]\\
    &\overset{*}{=}&E_\beta\left(-\frac{1}{2}\sum_{i,j=1}^n \lambda_i \lambda_j \langle M^{\mathbbm{H}}_{-} \mathds{1}_{[0,t_i)}, M^{\mathbbm{H}}_{-} \mathds{1}_{[0,t_j)} \rangle \right) \end{eqnarray*}
    where in $*$ we used Example 5.5 of \cite{BCDS23} and $f(\cdot)$ as before.
    
    \item It is evident that, when $m$ and $p$ are fixed, gFHp is determined solely by its covariance structure. Therefore, $X^{\mathbb{H},A,B}$ serves as an example of a stochastic process defined purely by its first and second moments, which is characteristic of Gaussian processes. This characteristic of gFHp was known for the class of ggBm; see \cite{MM09} and the reference within.

\end{enumerate}
\end{remark}

\begin{remark}
The real-valued gFHp we have introduced above may be extended to higher dimensions. There are two alternative and distinct methods to obtain a gFHp in higher dimensions. 
\begin{itemize}
    \item[1.] We may simply take independent copies of the real-valued gFHp and obtain a process with independent coordinates.
    \item[2.] Alternatively, we may start with the Hilbert space $L^2_d$ of vector-valued square integrable functions on $\mathbb{R}$ and the corresponding nuclear triple: $S_d\subset L^2_d\subset S'_d$. The corresponding gFHp associated is now a process with values in $\mathbb{R}^d$ but without independent coordinates due to the lack of the semi-group property of the function $_m\Psi_p$ in general. 
\end{itemize}
\end{remark}

\subsection{Representations}

The following proposition provides two alternative representations of the generalized Fox-$H$ processes in terms of other known processes. We denote by $Y_{A,B}$ the positive random variable with density in the class F$H$dam, i.e., $Y_{A,B}\in\mathcal{X}$, see Lemma~\ref{lem:HFiniteDensities}. More specifically, the probability density function of $Y_{A,B}$ is given by Equation~\eqref{eq:densitiesFiniteMoments}.
\begin{proposition} \label{Prop:SubordRepresen}
Let $Y_{A,B}$ be a random variable in $\mathcal{X}$ with density $H^{m,0}_{p,m}(\cdot)/K$, and $B^{\mathbbm{H}}$ the fractional Brownian motion with Hurst parameter $\mathbbm{H} \in (0,1)$, both defined on the same probability space $(\Omega, \mathcal{A}, \mathds{P})$. Then, provided that $Y_{A,B}$ and $B^{\mathbbm{H}}$ are independent, we have 

\begin{equation}\label{eq: Distribu Equalities}
     \{ X^{\mathbbm{H}, A, B}_t, \, t\geq 0\}\overset{\mathrm{f.d.d.}}{=\joinrel=\joinrel=}\{ \sqrt{Y_{A,B}}B^{\mathbbm{H}}_t , \, t\geq 0\} \overset{\mathrm{f.d.d.}}{=\joinrel=\joinrel=}\{ B^{\mathbbm{H}}_{t (Y_{A,B})^{1/(2\mathbbm{H})}}, \, t\geq 0\}.  \end{equation}
\end{proposition}

\begin{proof}
    Let $n \in \N$, $t_1,\dots,t_n \in (0,\infty)$ and $\lambda \in \R^n$ be given. For $X^{\mathbbm{H}, A, B}=(X^{\mathbbm{H}, A, B}_{t_1}
        ,\dots,X^{\mathbbm{H}, A, B}_{t_n})$ and $B^{\mathbbm{H}}=(B^{\mathbbm{H}}_{t_1}
        ,\dots,B^{\mathbbm{H}}_{t_n})$, we have
    \begin{eqnarray*}
        \mathds{E}_{\mu_{\Psi}}\bigl[\mathrm{e}^{\mathrm{i}(\lambda,X^{\mathbbm{H}, A, B})}\bigr]&\overset{*}{=}& \frac{1}{K}\, _m\!\Psi_p\left[ \genfrac{}{}{0pt}{}{(b_i+\beta_i,\beta_i)_{1,m}}{(a_j+\alpha_j,\alpha_j)_{1,p}} \middle| -\frac{ (\lambda, \Sigma_{\mathbbm{H},n} \lambda) }{2} \right] \\&=& \frac{1}{K}\int_0^\infty \mathrm{e}^{-\frac{s}{2}(\lambda,\Sigma_{\mathbbm{H},n} \lambda)}H^{m,0}_{p,m}(s)\,\mathrm{d}s \\ 
        &\overset{**}{=}&\frac{1}{K}\int_0^\infty \mathbb{E}\left[\mathrm{e}^{\mathrm{i}\sqrt{s}(\lambda,B^\mathbb{H})}\Big|Y_{A,B}=s\right]  H^{m,0}_{p,m}(s)\,\mathrm{d}s \\ 
        &=&\mathds{E}_{\mathds{P}}\bigl[\mathrm{e}^{\mathrm{i}(\lambda,\sqrt{Y_{A,B}}B^{\mathbbm{H}})}\bigr]
    \end{eqnarray*}
    where in $*$ the matrix $\Sigma_{\mathbbm{H},n}$ is given in Proposition~\ref{prop:CFDensityMomentsCovariance}-1, while in $**$ the measure $\mu_{B^\mathbbm{H}}$ represents the image measure of the fBm.\\
    The proof of the second equality in Equation~\eqref{eq: Distribu Equalities}, let $\theta=(\theta_1,\dots,\theta_n)\in\mathbb{R}^n$, be given. Using the conditional expectation, we obtain
\begin{eqnarray*}
&&\mathbb{E}_\mathbbm{P}\left[\exp\left(\mathrm{i}\sum_{k=1}^{n}\theta_{k}B^{\mathbb{H}}
_{t_{k}(Y_{A,B})^{1/(2\mathbbm{H})}}\right)\right] \\ 
&=&\mathbb{E}_\mathbbm{P}\left[\mathbb{E}\left[\exp\left(\mathrm{i}\sum_{k=1}^{n}\theta_{k}B^{\mathbb{H}}_{t_{k}(Y_{A,B})^{1/(2\mathbbm{H})}}\right)\Bigg|Y_{A,B}\right]\right]\\
&=&\frac{1}{K}\int_0^\infty\mathbb{E}\left[\exp\left(\mathrm{i}\sum_{k=1}^{n}\theta_{k}B^{\mathbb{H}}_{t_{k}s^{1/(2\mathbbm{H})}}\right)\Bigg|Y_{A,B}=s\right]H_{p,m}^{m,0}(s)\,\mathrm{d}s\\
&=&\frac{1}{K}\int_0^\infty\mathbb{E}\left[\exp\left(\mathrm{i}\sum_{k=1}^{n}\theta_{k}s^{1/2}B^{\mathbb{H}}_{t_{k}}\right)\Bigg|Y_{A,B}=s\right]H_{p,m}^{m,0}(s)\,\mathrm{d}s\\ &=&\mathbb{E}_\mathbbm{P}\left[\exp\left(\mathrm{i}\sum_{k=1}^{n}\theta_{k}\sqrt{Y_{A,B}}B^{\mathbb{H}}_{t_{k}}\right)\right],
\end{eqnarray*}
where, in the third equality, we have used the $\mathbb{H}$-self-similarity of fBm. This concludes the proof. 
\end{proof}

The aforementioned representations allow us to establish the following properties for the gFHp.
\begin{corollary}
For $\mathbbm{H}\in(0,\frac{1}{2}) \cup (\frac{1}{2},1)$, the generalized Fox-$H$ process is not a semimartingale. In addition, $X^{\mathbbm{H}, A, B}_t, t \geq 0$, cannot have finite variation on $[0,1]$ and, by scaling and the stationarity of increments, on any interval. 
\end{corollary}

\begin{proof}
    For $\mathbbm{H}\neq 1/2$ we can get the results by noting that the quadratic variation behaves as for the fractional Brownian motion. More explictly, denoting by $\pi_n, n \in \N$, the set of $n$ elements partitioning $[0,1]$, we have
    \begin{eqnarray*}   &&\mathds{E}\left(\sum_{ \pi_n}(X^{\mathbbm{H}, A, B}_{t_{i+1}}-X^{\mathbbm{H}, A, B}_{t_i})^2\right)=\mathds{E}\left(\sum_{ \pi_n}(\sqrt{Y_{A,B}}B^{\mathbbm{H}}_{t_{i+1}}-\sqrt{Y_{A,B}}B^{\mathbbm{H}}_{t_i})^2\right)\\
    &=& \int_0^\infty \mathds{E}\left(y\left(\sum_{ \pi_n}(B^{\mathbbm{H}}_{t_{i+1}}-B^{\mathbbm{H}}_{t_i})^2\right) \middle| \, Y_{A,B}=y\right)\mathds{P}_{Y_{A,B}}(\mathrm{d}y)\\
    &=&\mathds{E}(Y_{A,B})\sum_{\pi_n} (t_{i+1} - t_i)^{2\mathbbm{H}-1}.  \end{eqnarray*}
    Then, the result can be derived through the same reasoning applied to fractional Brownian motion, see p.71 in \cite{Mishura08}.
\end{proof}

\subsection{The Generalized Fox-\texorpdfstring{$H$}{H} Noise Process}
In the following, we use the characterization Theorem~\ref{thm:StransIsomor} and Theorem~\ref{thm:StransConverg} to establish the existence of the generalized Fox-$H$ noise process $N_t^{\mathbbm{H}, A, B}$, $t\ge0$, as an element in $(\mathcal{S})_{\mu_\Psi}^{-1}$ by giving the $S_{\mu_{\Psi}}$-transform.
More specifically, using the $S_{\mu_{\Psi}}$-transform of $X^{\mathbbm{H}, A, B}_t$, we prove the existence of the time derivative of $X_t^{\mathbbm{H}, A, B}$, in a distributional sense, as an element in $(\mathcal{S})^{-1}_{\mu_\Psi}$.\\
For $l,k \in \mathbb{N}$ and $\xi \in U_{l,k}$, the $S_{\mu_{\Psi}}$-transform of $X^{\mathbbm{H}, A, B}_t$ is given by
	
	\begin{eqnarray}
		\bigl(S_{\mu_{\Psi}}X^{\mathbbm{H}, A, B}_t\bigr)(\xi)&=&\ell^{-1}_{\mu_{\Psi}}(\xi)\int_{\mathcal{S}'}\langle \omega ,M_{-}^{\mathbbm{H}}\mathbbm{1}_{[0,t)}\rangle \mathrm{e}^{\langle \omega, \xi\rangle }\, \mathrm{d}\mu_{\Psi}(\omega) \label{momentlaplace}\\
		&=&\ell^{-1}_{\mu_{\Psi}}(\xi) \int_{\mathcal{S}'} \frac{\partial}{\partial s} \mathrm{e}^{\langle \omega, \xi\rangle + s\langle \omega,M_{-}^{\mathbbm{H}}\mathbbm{1}_{[0,t)}  \rangle}\Big|_{s=0}\, \mathrm{d}\mu_{\Psi}(\omega)\\
		&=&\ell^{-1}_{\mu_{\Psi}}(\xi) \frac{\partial}{\partial s}\int_{\mathcal{S}'}  \mathrm{e}^{\langle \omega, \xi + sM_{-}^{\mathbbm{H}}\mathbbm{1}_{[0,t)}  \rangle}\, \mathrm{d}\mu_{\Psi}(\omega)\Big|_{s=0},
	\end{eqnarray}
	where we can exchange the integral and the derivative because the integral in (\ref{momentlaplace}) is finite by the Cauchy-Schwarz inequality, $X^{\mathbbm{H}, A, B}_t\in L^2(\mu_\Psi)$.
	Recalling 
    \begin{eqnarray*}
    &&\int_{\mathcal{S}'}  \e^{\langle \omega, \xi + sM_{-}^{\mathbbm{H}}\mathbbm{1}_{[0,t)}  \rangle} \mu_{\Psi}(\mathrm{d}\omega)\\
    &&= \frac{1}{K} \,\,_m\Psi_p\left[ \genfrac{}{}{0pt}{}{(b_i+\beta_i, \beta_i )_{1,m}}{(a_j+\alpha_j, \alpha_j )_{1,p} }\Bigg|\frac{1}{2}\langle\xi + sM_{-}^{\mathbbm{H}}\mathbbm{1}_{[0,t)},\xi + sM_{-}^{\mathbbm{H}}\mathbbm{1}_{[0,t)}\rangle \right],
    \end{eqnarray*}
    we compute its derivative using Remark~\ref{rem:GeneralizedWrightDerivatives}
	\begin{eqnarray*}&& \frac{\partial}{\partial s}\int_{\mathcal{S}'}  \e^{\langle \omega, \xi + sM_{-}^{\mathbbm{H}}\mathbbm{1}_{[0,t)}\rangle}\, \mathrm{d}\mu_{\Psi}(\omega)\\
    &=& \, _m\Psi_p\left[ \genfrac{}{}{0pt}{}{(b_i+2\beta_i, \beta_i )_{1,m}}{(a_j+2\alpha_j, \alpha_j )_{1,p} }\Bigg| \frac{1}{2}|\xi + sM_{-}^{\mathbbm{H}}\mathbbm{1}_{[0,t)} |^2_0 \right]\frac{\Big( s |M_{-}^{\mathbbm{H}}\mathbbm{1}_{[0,t)}|^2_0 + \langle \xi, M_{-}^{\mathbbm{H}}\mathbbm{1}_{[0,t)}  \rangle \Big)}{K}.\end{eqnarray*}
	Hence, we get
	
	\begin{eqnarray*}\label{StransfGWGBM}
    &&(S_{\mu_{\Psi}}X^{\mathbbm{H}, A, B}_t)(\xi)=\ell^{-1}_{\mu_{\Psi}}(\xi) \frac{\partial}{\partial s}\int_{\mathcal{S}'}  \e^{\langle \omega, \xi + sM_{-}^{\mathbbm{H}}\mathbbm{1}_{[0,t)}  \rangle}\, \mathrm{d}\mu_{\Psi}(\omega)\Big|_{s=0}\\
	&=&  \frac{1}{\,\,_m\Psi_p( \frac{1}{2}|\xi|^2_0)}\, _m\Psi_p\left[ \genfrac{}{}{0pt}{}{(b_i+2\beta_i, \beta_i )_{1,m}}{(a_j+2\alpha_j, \alpha_j )_{1,p} }\Bigg| \frac{1}{2}|\xi |^2_0 \right] \langle \xi, M_{-}^{\mathbbm{H}}\mathbbm{1}_{[0,t)}  \rangle \\
    &\overset{*}{=}& \frac{1}{\,\,_m\Psi_p( \frac{1}{2}|\xi|^2_0)}\, _m\Psi_p\left[ \genfrac{}{}{0pt}{}{(b_i+2\beta_i, \beta_i )_{1,m}}{(a_j+2\alpha_j, \alpha_j )_{1,p} }\Bigg| \frac{1}{2}|\xi |^2_0 \right] \left(\int_{0}^tM_{+}^{\mathbbm{H}}\xi(x)\mathrm{d}x \right), \end{eqnarray*}
	where in $*$ we used Equations~(5.16) and (5.17) in \cite{SKM93}.

	Now we establish the differentiability in $(\mathcal{S})^{-1}_{\mu_{\Psi}}$ of the gFHp, $X^{\mathbbm{H},A,B}_t,t\geq 0$, that is, the noise of gFHp is an element in $(\mathcal{S})^{-1}_{\mu_{\Psi}}$.
	\begin{theorem}\label{thm:noiseExiste}
		Let the assumptions of Lemma \ref{lem:HFiniteDensities} hold and $\mathbbm{H} \in (0,1)$.
		We define the gFHp in $(\mathcal{S})^{-1}_{\mu_{\Psi}}$ as,
		\[ 
        N_t^{\mathbbm{H}, A, B}(\omega):=\lim_{h \to 0} \frac{X^{\mathbbm{H}, A, B}_{t+h}(\omega)- X^{\mathbbm{H}, A, B}_t(\omega) }{h}, \quad \mu_{\Psi}\mbox{-} \text{a.a.}\; \omega \in \mathcal{S}',
        \]
		and exists a $l_\mathbbm{H} \in \N$ such that, for every $\xi \in U_{l,k},l\geq l_\mathbbm{H}$, we have		
		\[  
        \big(S_{\mu_{\Psi}} N_t^{\mathbbm{H}, A, B}\big)(\xi)=\frac{1}{\,\,_m\Psi_p( \frac{1}{2}|\xi|^2_0)}\, _m\Psi_p\left[ \genfrac{}{}{0pt}{}{(b_i+2\beta_i, \beta_i )_{1,m}}{(a_j+2\alpha_j, \alpha_j )_{1,p} }\Bigg| \frac{1}{2}|\xi |^2_0 \right](M_{+}^{\mathbbm{H}} \xi)(t).
        \]
	\end{theorem}
	
	\begin{proof}
		For $t > 0$ and $n \in \mathbb{N}$, we define the sequence 
		\[N^{\mathbbm{H}, A, B}_{t,n}(\cdot):=\frac{X^{\mathbbm{H}, A, B}_{t+h_n}(\cdot)- X^{\mathbbm{H}, A, B}_{t}(\cdot)}{ h_n} \in L^2(\mu_{\Psi}) \subset (\mathcal{N})^{-1}_{\mu_{\Psi}}\]
		where $\{h_n\}_{n\in \mathbb{N}}$ are such that $h_n \to 0$ for $n \to \infty$.\\
        By Lemma~4.4 in \cite{BCM24}, there exists a $l_\mathbbm{H}\in \mathbb{N}$ such that
		\[ h_n^{-1}\int_{t}^{t+h_n} |M_{+}^{\mathbbm{H}} \xi(x)| dx\leq \max_{x \in \mathbb{R}}(|M_{+}^{\mathbbm{H}} \xi(x)|)\leq c_{\mathbbm{H}}|\xi|_{l_\mathbbm{H}}.\]
		Hence, we apply the $S_{\mu_{\Psi}}$-transform for $\xi \in U_{l_{\mathbbm{H}},0}$
		\begin{eqnarray*}
			\big(S_{\mu_{\Psi}} N^{\mathbbm{H}, A, B}_{t,n} \big)(\xi)&=&\frac{1}{h_n}\Big( (S_{\mu_{\Psi}}X^{\mathbbm{H}, A, B}_{t+h_n})(\xi) - (S_{\mu_{\Psi}} X^{\mathbbm{H}, A, B}_{t})(\xi)\Big)  \\
			&=&\frac{1}{\,\,_m\Psi_p( \frac{1}{2}|\xi|^2_0)}\, _m\Psi_p\left[ \genfrac{}{}{0pt}{}{(b_i+2\beta_i, \beta_i )_{1,m}}{(a_j+2\alpha_j, \alpha_j )_{1,p} }\Bigg| \frac{1}{2}|\xi |^2_0 \right] \\
            &&\times\left( \frac{\int_{t}^{t+h_n}M_{+}^{\mathbbm{H}}\xi(x)\mathrm{d}x}{h_n} \right),
		\end{eqnarray*}
		and from the analyticity of $\ell^{-1}_{\mu_{\Psi}}(\cdot) $ around zero, there exist $k\in \mathbb{N}$ and $l\geq l_\mathbbm{H}$ such that $\ell^{-1}_{\mu_{\Psi}}(\xi)\in (M_1,M_2)$ for $\xi \in U_{l,k}\subset U_{l_\mathbbm{H},0}$ where $M_1 \in (0,1)$ and $M_2 \in (1,\infty)$. Furthermore, the analyticity of $_m\Psi_p(\cdot)$ around zero ensures that its derivative around zero is bounded. Hence, there exists a $K>0$
		
		\[ \left|\frac{1}{\,\,_m\Psi_p( \frac{1}{2}|\xi|^2_0)}\, _m\Psi_p\left[ \genfrac{}{}{0pt}{}{(b_i+2\beta_i, \beta_i )_{1,m}}{(a_j+2\alpha_j, \alpha_j )_{1,p} }\Bigg| \frac{1}{2}|\xi |^2_0 \right]  \right| < K \quad \text{ for } \xi \in U_{l,k}.\] 
		
		Finally, we have that, for each $n \in \mathbb{N}$ and $\xi \in U_{l,k} $
		
		\[  |\big(S_{\mu_{\Psi}} N^{\mathbbm{H}, A, B}_{t,n} \big)(\xi)|<K c_\mathbbm{H} | \xi |_{l} .\]
		By Section 4.2 in \cite{BCM24} we have that $(M_{+}^{\mathbbm{H}} \xi)(\cdot)\in C^{\infty}(\mathbb{R})$, thus
		
		\begin{eqnarray*} 
			&&\lim_{n \to \infty}\big(S_{\mu_{\Psi}} N^{\mathbbm{H}, A, B}_{t,n} \big)(\xi)\\
            &=& \frac{1}{\,\,_m\Psi_p( \frac{1}{2}|\xi|^2_0)}\, _m\Psi_p\left[ \genfrac{}{}{0pt}{}{(b_i+2\beta_i, \beta_i )_{1,m}}{(a_j+2\alpha_j, \alpha_j )_{1,p} }\Bigg| \frac{1}{2}|\xi |^2_0 \right]\lim_{n \to \infty} \left(\frac{\int_{t}^{t+h_n}M_{+}^{\mathbbm{H}}\xi(x)\mathrm{d}x}{h_n} \right)\\
			&=&\frac{1}{\,\,_m\Psi_p( \frac{1}{2}|\xi|^2_0)}\, _m\Psi_p\left[ \genfrac{}{}{0pt}{}{(b_i+2\beta_i, \beta_i )_{1,m}}{(a_j+2\alpha_j, \alpha_j )_{1,p} }\Bigg| \frac{1}{2}|\xi |^2_0 \right] \left(M_{+}^\mathbbm{H}\xi\right)(t) .\\
		\end{eqnarray*}
		Thus, $\left(\big(S_{\mu_{\Psi}} N^{\mathbbm{H}, A, B}_{t,n} \big)(\xi) \right), n \in \N$, is a Cauchy sequence for each $\xi \in U_{l,k}, l \geq l_{\mathbbm{H}}$, $k \in \N$.\\
		By the fact that $\big(S_{\mu_{\Psi}} N^{\mathbbm{H}, A, B}_{t,n} \big)(\xi)$ is holomorphic on $U_{l,k}$ and is finite for each $\xi \in U_{l,k}$, we can apply Theorem \ref{thm:StransConverg} for the convergence of $N^{\mathbbm{H}, A, B}_{t,n}$ to an element in $(\mathcal{S})_{\mu_\Psi}^{-1}$ that we denote by $N^{\mathbbm{H}, A, B}_t$.
	\end{proof}

\section{Applications\label{sec:applications}}
In the following section, we give some applications using the gFHp. We start by showing the existence of the occupation density of $X^{\mathbbm{H}, A, B}_t$, $t\ge0$, at $x \in \mathbb{R}$, employing the criteria of Berman, see \cite{Ber69}, cf.~Subsection~\ref{subsec:gFHp-Ltimes}.

\subsection{Existence of Local Times for gFHp\label{subsec:gFHp-Ltimes}}

For the reader's convenience, we recall the notion of occupation measure
and occupation density. Let $f:I\longrightarrow\mathbb{R}$, $I$
a Borel set in $[0,1]$, be a measurable function and define, for
any set $B\in\mathcal{B}(\mathbb{R})$, the occupation measure $\mu_{f}$
on $I$ by
\[
\mu_{f}(B):=\int_{I}1\!\!1_{B}(f(s))\, \mathrm{d}s.
\]
Interpreting $[0,1]$ as a ''time set'' is the ''amount of
time spent by $f$ in $B$ during the time period $I$''. We say
that $f$ has an occupation density over $I$ if $\mu_{f}$ is absolutely
continuous with respect to the Lebesgue measure $\lambda$ and we denote this density by $L^{f}(\cdot,I)$. In explicit, for any $x\in\mathbb{R}$,
\[
L^{f}(x,I)=\frac{d\mu_{f}}{\mathrm{d}\lambda}(x).
\]
Thus, we have
\[
\mu_{f}(B)=\int_{I}1\!\!1_{B}(f(s))\, \mathrm{d}s=\int_{B}L^{f}(x,I)\, \mathrm{d}x.
\]
A continuous stochastic process $X$ has an occupation density on
$I$ if, for almost all $w\in\Omega$, $X(w)$ has an occupation density
$L^{X}(\cdot,I)$, also called local time of $X$, see Berman \cite{Ber69}.

The criteria for the existence of local times for real-valued stochastic processes
$X$ are due to Berman \cite[Section~3]{Ber69}. More precisely,
a stochastic process $X$ admits a $\lambda$-square integrable local times if and only if
\begin{equation}
\int_{\mathbb{R}}\int_{0}^{1}\int_{0}^{1}\mathbb{E}\big(\e^{\mathrm{i}\theta(X(t)-X(s))}\big)\, \mathrm{d}s\, \mathrm{d}t\,
\mathrm{d}\theta<\infty.\label{Existence-LT}
\end{equation}
In the next, we show that (\ref{Existence-LT}) is fulfilled if $X$ is the gFHp $X^{\mathbbm{H}, A, B}$ .

\begin{theorem}
\label{thm:gBm_LT} Let $2b_j + \beta_j>0$ for $j=1,\dots,m$, then the gFHp $X^{\mathbbm{H}, A, B}$ admits a $\lambda$-square
integrable local time $L^{X^{\mathbbm{H}, A, B}}(\cdot,I)$ $\mu_\Psi$-almost surely.
\end{theorem}

\begin{proof}
The characteristic function of the increments
of gFHp $X^{\mathbbm{H}, A, B}$ is given by

\begin{eqnarray*}
\mathbb{E}\big(\mathrm{e}^{ix(X^{\mathbbm{H}, A, B}_t-X^{\mathbbm{H}, A, B}_s)}\big) = \frac{1}{K}\,  _m \Psi_p\left(-\frac{x^{2}}{2}|t-s|^{2\mathbbm{H}}\right).
\end{eqnarray*}
Using the LT representation of $_m \Psi_p(-\cdot)$, denoting $a:=(2)^{-1}|t-s|^{2\mathbbm{H}}$ we have

\begin{eqnarray*}
   \int_{\mathbb{R}}\, _m \Psi_p\left(-\frac{x^{2}}{2}|t-s|^{2 \mathbbm{H}}\right)\mathrm{d}x&=&\int_{\mathbb{R}} \int_0^\infty e^{-ax^2 r}H^{m,0}_{p,m}(r)\mathrm{d}r\mathrm{d}x\\
    &\overset{*}{=}&\int_0^\infty H^{m,0}_{p,m}(r)\int_{\mathbb{R}}e^{-ax^2 r}\mathrm{d}x   \mathrm{d}r\\
    &=&\sqrt{\frac{\pi}{a}}\int_0^\infty 
 r^{-1/2}H^{m,0}_{p,m}(r) \mathrm{d}r\\
 &=&\sqrt{\frac{\pi}{a}} \big(\mathcal{M}H^{m,0}_{p,m}\big) (1/2)<\infty,
\end{eqnarray*}
where we use Fubini in $*$ and Theorem~2.2 in \cite{SaiKil} together with $2b_j + \beta_j>0,j=1,\dots,m$, to ensure that the constant $\bigl(\mathcal{M}H^{m,0}_{p,m}\bigr) (1/2)$ identifies the Mellin transform of $H^{m,0}_{p,m}$ evaluated in $1/2$, see \cite{SaiKil}.\\
To conclude the proof, we need to show the finiteness of the following integral
\begin{eqnarray*}
\int_{0}^{1}\int_{0}^{1}\frac{1}{|t-s|^{\mathbbm{H}}}\, \mathrm{d}s\, \mathrm{d}t & = & 2 \int_{0}^{1}\int_{0}^{t}\frac{1}
{|t-s|^{\mathbbm{H}}}\, \mathrm{d}s\, \mathrm{d}t\\
 & = & \frac{2}{1-\mathbbm{H}}\int_{0}^{1}t^{1-\mathbbm{H}}\, \mathrm{d}t\\
 & = & \frac{2}{(2-\mathbbm{H})(1-\mathbbm{H})}.\\
\end{eqnarray*} We showed the claim using Equation~\eqref{Existence-LT}.  
\end{proof}

\begin{remark}
As a consequence of the existence of the local time $L^{X^{\mathbbm{H},A,B}}(\cdot,I)$,
we obtain the occupation formula
\[
\int_{I}g(X^{\mathbbm{H},A,B}_s)\, \mathrm{d}s=\int_{\mathbb{R}}g(x)L^{X^{\mathbbm{H},A,B}}(x,I)\, \mathrm{d}x,\; a.s.-\omega \in \mathcal{S}',
\]
for each $g: \mathbb{R}\to\mathbb{R}$ measurable and bounded function. Moreover, using the informal representation of local times as 
\begin{equation}\label{eq:LT-informal}
L^{X^{\mathbbm{H},A,B}}(x,[0,t])(\omega)=\int_0^t \delta_x(\langle \omega, M^{\mathbbm{H}}_{-} \mathds{1}_{{[0,t)}} \rangle) \,\mathrm{d}s
\end{equation}
it follows from Proposition 5.5 in \cite{BCDS24} and its proof, together with Theorem~\ref{thm:CharIntegrableMap}, that $L^{X^{\mathbbm{H},A,B}}(x,[0,t])(\cdot) \in (\mathcal{S})^{-1}_{\mu_\Psi}$.
\end{remark}

\subsection{Anomalous Diffusion}
Now we would highlight the different kinds of diffusion that gFHp can exhibit. To differentiate between them, the variance growth over time is compared with linear growth. A faster (respectively, similar, or slower) increase in the process's variance indicates super-diffusion (respectively, diffusion, and sub-diffusion).
The gFHp exhibits different diffusion behaviors depending on the Hurst parameter $\mathbbm{H}$, as stated in the following proposition.
\begin{proposition}
    The gFHp has sub-diffusion for $\mathbbm{H} \in (0,1/2)$, normal diffusion if $\mathbbm{H}=1/2$ and super-diffusion for $\mathbbm{H} \in (1/2,1)$.
\end{proposition}
\begin{proof}
    The claim is obtained by comparing the asymptotic behavior of the gFHp variance with respect to linear growth in time; see Proposition~\ref{prop:CFDensityMomentsCovariance}-3.
\end{proof}

\section{Outlook and Discussion}
We applied the generalized Wright analysis from \cite{BCDS24} to define a new class of processes $X^{\mathbbm{H},A,B}_t$, $t \geq 0$, called generalized Fox-$H$ processes, and conducted a comprehensive examination of their essential properties. A particularly remarkable aspect of these processes is their representation in terms of fBm $ B^\mathbbm{H}$ and an independent positive random variable $ Y_{A,B} $ that possesses a F$H$dam density.  The gFHp has the same f.d.d.~as the time-changed fBm, that is, $B^\mathbbm{H}_{t Y_{A,B}^{1/2\mathbbm{H}}}$ and also as $\sqrt{Y_{A,B}}B^\mathbbm{H}_t$, $t\ge0$. These representations of the gFHp allow us to derive the path properties from fBm, as well as to perform simulations. It has stationary increments, is $\mathbbm{H}$-self-similar, H{\"o}lder continuous of order $\gamma <\mathbbm{H}$. This class of processes includes well-known processes like Bm, fBm, ggBm, among others. 
Further directions for investigation within this framework include:
\begin{itemize}
    \item[1.] SPDE driven by gFHp, for example  
$$
\partial_t u_t = Au_t + u_t \diamond \mathrm{d}X_t. 
$$
where $A$ is a linear symmetric compact operator;
    \item[2.] The existence of the Green function for the time-fractional heat equation using the gFHp;
    \item[3.] The representation of gFHp via Ornstein-Uhlenbeck processes;
    \item[4.] The local times (or self-intersection local times) of $X^{\mathbbm{H}, A, B}$ in higher dimensions and study their properties;
    \item[5.] Malliavin calculus and its applications to quantum fields.
\end{itemize}

\vspace{1cm}

\bibliographystyle{alpha}
\bibliography{Hanalysis}
\end{document}